\documentclass[11pt]{amsart}
\usepackage{amsmath,amssymb,amsthm,amsxtra} 
\usepackage{stmaryrd}
\usepackage{times}
\usepackage[margin=1.4in]{geometry}
\usepackage[linktocpage=true, hidelinks, colorlinks=true, allcolors=blue]{hyperref}
\usepackage{comment}
\usepackage{scalerel}

\makeatletter
\def\l@subsection{\@tocline{2}{0pt}{2.5pc}{5pc}{}}
\renewcommand\tocchapter[3]{%
  \indentlabel{\@ifnotempty{#2}{\ignorespaces#2.\quad}}#3%
}
\newcommand\@dotsep{4.5}
\def\@tocline#1#2#3#4#5#6#7{\relax
  \ifnum #1>\c@tocdepth 
  \else
    \par \addpenalty\@secpenalty\addvspace{#2}%
    \begingroup \hyphenpenalty\@M
    \@ifempty{#4}{%
      \@tempdima\csname r@tocindent\number#1\endcsname\relax
    }{%
      \@tempdima#4\relax
    }%
    \parindent\z@ \leftskip#3\relax \advance\leftskip\@tempdima\relax
    \rightskip\@pnumwidth plus1em \parfillskip-\@pnumwidth
    #5\leavevmode\hskip-\@tempdima{#6}\nobreak
    \leaders\hbox{$\m@th\mkern \@dotsep mu\hbox{.}\mkern \@dotsep mu$}\hfill
    \nobreak
    \hbox to\@pnumwidth{\@tocpagenum{#7}}\par
    \nobreak
    \endgroup
  \fi}
\makeatother
\AtBeginDocument{%
\makeatletter
\expandafter\renewcommand\csname r@tocindent0\endcsname{0pt}
\makeatother
}
\def\l@subsection{\@tocline{2}{0pt}{2.5pc}{5pc}{}}

\newtheorem{thm}{Theorem}[section]
\newtheorem{lemma}[thm]{Lemma}
\newtheorem{proposition}[thm]{Proposition}

\newtheorem{definition}[thm]{Definition}
\newtheorem{corollary}[thm]{Corollary}

\newtheorem{notation}[thm]{Notation}

\newtheorem*{maintheorem*}{Main Theorem}
\newtheorem*{theorem*}{Theorem}
\newtheorem*{theorem1*}{Theorem 1}
\newtheorem*{theorem2*}{Theorem 2}
\newtheorem*{theorem3*}{Theorem 3}
\newtheorem*{corollary*}{Corollary}

\newcommand{\p}{\mathbb{P}}
\newcommand{\q}{\mathbb{Q}}

\newcommand{\cof}{\mathrm{cof}}
\newcommand{\dom}{\mathrm{dom}}

\newcommand{\col}{\mathrm{Col}}

\newcommand{\res}{\upharpoonright}
\newcommand{\cl}{\text{cl}}

\newcommand{\sk}{\mathrm{Sk}}

\long\def\comment#1{}

\begin{document}

\title[A Strongly Non-Saturated Aronszajn Tree Without Weak Kurepa Trees]{A Strongly Non-Saturated Aronszajn Tree \\ Without Weak Kurepa Trees}

\author{John Krueger and \v{S}\'{a}rka Stejskalov\'{a}}

\address{John Krueger, Department of Mathematics, 
	University of North Texas,
	1155 Union Circle \#311430,
	Denton, TX 76203, USA}
\email{john.krueger@unt.edu}

\address{\v{S}\'{a}rka Stejskalov\'{a}, Department of Logic, 
Charles University, Celetn\'{a} 20, Prague 1, 
116 42, Czech Republic}
\email{sarka.stejskalova@ff.cuni.cz}



\date{June 26, 2025}

\begin{abstract}
Assuming the negation of Chang's conjecture, there is a c.c.c.\ forcing which 
adds a strongly non-saturated Aronszajn tree. 
Using a Mahlo cardinal, we construct a model in which there 
exists a strongly non-saturated Aronszajn tree and the 
negation of the Kurepa hypothesis is c.c.c.\ indestructible.
For any inaccessible cardinal $\kappa$, 
there exists a forcing poset which is Y-proper and $\kappa$-c.c., 
collapses $\kappa$ to become $\omega_2$, and adds a 
strongly non-saturated Aronszajn tree. 
The quotients of this forcing in intermediate extensions are indestructibly 
Y-proper on a stationary set with respect to any Y-proper forcing extension. 
As a consequence, we prove from an inaccessible cardinal 
that the existence of a strongly 
non-saturated Aronszajn tree is consistent with the non-existence of a 
weak Kurepa tree. 
Finally, we prove from a supercompact cardinal that the existence of a 
strongly non-saturated Aronszajn tree is consistent with 
two-cardinal tree properties such as the indestructible guessing model principle.
\end{abstract}

\maketitle

\tableofcontents

\section{Introduction}

An Aronszajn tree is \emph{saturated} if every family of uncountable 
downwards closed subtrees of it which is almost disjoint 
has cardinality less than $\omega_2$. 
By almost disjoint we mean that 
any two members of the family have countable intersection. 
This idea was introduced by 
K\"onig, Larson, Moore, and Veli\v{c}kovi\'{c} in the context of 
a study of the consistency strength of Moore's five element basis theorem 
(\cite{moorebounding}, \cite{moorefive}). 
A classic example of an Aronszajn tree which is non-saturated is due to 
Todor\v{c}evi\'{c}, who showed that if $U$ is a special Aronszajn tree and 
$T$ is a Kurepa tree, then the tree product $U \otimes T$ is a special Aronszajn 
tree which has no base of subtrees of cardinality less than $\omega_2$ (\cite{baumgartnerbase}). 
Moore asked whether the existence of a non-saturated 
Aronszajn tree implies the existence of a Kurepa tree (\cite{moorestructural}). 
This problem was solved by the authors by showing that 
if $\kappa$ is an inaccessible cardinal and $T$ is a free Suslin tree, then there is a generic 
extension in which $\kappa = \omega_2$, the square $T \otimes T$ 
is a non-saturated Aronszajn tree, 
and there does not exist a Kurepa tree (\cite{KS}).

Martinez Mendoza and the author introduced a stronger form of non-saturation 
for an Aronszajn tree $T$. 
Given subtrees $U$ and $W$ of $T$, we say that $U$ and $W$ are \emph{strongly almost disjoint} 
if $U \cap W$ is a finite union of countable chains 
(equivalently, $U \cap W$ does not contain an infinite antichain). 
A stronger property that 
we will use is that $U \cap W$ is contained in the downward closure of a 
finite subset of $T$, which we refer to as $U \cap W$ being \emph{finitely generated}. 
The tree $T$ is \emph{strongly non-saturated} if there exists a family of $\omega_2$-many 
strongly almost disjoint uncountable downwards closed subtrees of $T$ (\cite{KM}). 
Note that if there exists such a tree, then \textsf{CH} is false.

The idea of a strongly non-saturated Aronszajn tree 
is a relative of that of a 
collection of uncountable subsets of $\omega_1$ with size $\omega_2$ 
such that the intersection of any two members of the collection is finite. 
The existence of such a family was proven 
to be consistent by Baumgartner \cite{almostdisjoint}. 
Indeed, if there is a strongly non-saturated Aronszajn tree, then there 
is such a collection.\footnote{Without loss of generality, assume that the 
strongly non-saturated Aronszajn tree $T$ is Hausdorff. 
Let $\{ U_\alpha : \alpha < \omega_2 \}$ be a strongly 
almost disjoint family of uncountable downwards closed subtrees of $T$. 
For each $\alpha < \omega_2$, define $W_\alpha$ to be the collection of all 
sets $\{ x, y, z \} \subseteq U_\alpha$ 
such that $y$ and $z$ are distinct immediate successors of $x$ in $U_\alpha$. 
Then for all $\alpha < \beta < \omega_2$, the fact that $U_\alpha \cap U_\beta$ 
contains no infinite antichain implies by a straightforward argument 
that $W_\alpha \cap W_\beta$ is finite. 
So $\{ W_\alpha : \alpha < \omega_2 \}$ is a strongly almost disjoint 
family of uncountable subsets of $[T]^3$.} 
The consistency of a strongly non-saturated Aronszajn tree 
was proven by Martinez Mendoza and the author by showing, 
using the $\rho$-function of Todor\v{c}evi\'{c} \cite{todorpartition} 
under the assumption of $\Box_{\omega_1}$, 
that there exists a c.c.c.\ forcing which adds an almost Kurepa 
Suslin tree whose square is strongly non-saturated (\cite{KM})

With this stronger version of non-saturation for Aronszajn trees 
at hand, it is natural to ask whether 
the existence of a strongly non-saturated Aronszajn tree 
is consistent with the non-existence of a Kurepa tree, or even 
with the non-existence of a weak Kurepa tree. 
The latter question was asked in \cite{KM}. 
The connection with the weak Kurepa hypothesis is that the existence of 
a weak Kurepa tree follows from \textsf{CH} and thus holds 
in the model of \cite{KS}. 
Both questions can be thought of 
as more ambitious versions of Moore's problem stated above 
which are suitable in the context of the negation of \textsf{CH}.

In order to solve these problems, we introduce two new forcing posets for 
adding a strongly non-saturated Aronszajn tree, both with finite conditions. 
The first forcing poset is c.c.c.\ assuming the negation of Chang's conjecture. 
For this forcing, the main idea is to use a weak kind of $\rho$-function to bound 
the intersection of the subtrees appearing in a condition. 
The existence of such a function follows from the negation of Chang's conjecture 
(\cite{todorchang}). 
Using this forcing poset in combination with a result of Jensen and Schlechta 
\cite{jensenschlechta}, 
we have the following theorem which provides a strong answer to the first question.

\begin{theorem1*}
	Assume that $\kappa$ is a Mahlo cardinal. 
	Then there is a generic extension of $L$ in which there exists a 
	strongly non-saturated Aronszajn tree and the negation of the 
	Kurepa hypothesis is c.c.c.\ indestructible.
\end{theorem1*}

For the second problem, for any inaccessible cardinal $\kappa$ 
we prove that there exists a forcing poset with finite conditions which adds a 
strongly non-saturated Aronszajn tree, collapses $\kappa$ to become $\omega_2$, 
and has quotients in intermediate extensions which do not add new cofinal 
branches to trees of height $\omega_1$. 
The conditions in this forcing consist of two parts: a working part which is a finite 
approximation of the generic tree and its subtrees, and a side condition which is a 
finite set of countable models. 
The interaction between the working part and the side condition is 
that whenever the indices of two of the subtrees from the working part 
are members of a model in the side condition, 
then the intersection of the subtrees is a subset of the model as well. 
Using this forcing together with a standard Silver factor analysis (\cite{silver}), 
we have the following theorem.

\begin{theorem2*}
	Assume that $\kappa$ is an inaccessible cardinal. 
	Then there exists a forcing poset which is proper, 
	collapses $\kappa$ to become $\omega_2$, adds a strongly 
	non-saturated Aronszajn tree, and forces the non-existence of 
	a weak Kurepa tree.
\end{theorem2*}

In order to prove that the quotient forcings do not add new cofinal branches, 
we use the concept of Y-properness which was recently introduced by 
Chodounsk\'{y} and Zapletal \cite{CZ}. 
This property implies not only not adding new cofinal branches to trees of 
height $\omega_1$, but also the stronger 
$\omega_1$-approximation property. 
As a result, we are able to prove much more than the negation of 
the weak Kurepa hypothesis. 
In fact, not only are the quotients of our forcing poset Y-proper, but they 
remain Y-proper on a stationary set after any further Y-proper forcing. 
This enables us to do additional forcing after the main forcing while still 
allowing for the usual factor analysis. 
In particular, we have the following theorem related to two-cardinal tree principles.

\begin{theorem3*}
	Assume that $\kappa$ is a supercompact cardinal. 
	Then there is a generic extension in which $\kappa = \omega_2$, there exists 
	a strongly non-saturated Aronszajn tree, and the indestructible 
	guessing model principle \textsf{IGMP} holds.
\end{theorem3*}

We note that \textsf{IGMP} 
follows from \textsf{PFA}, whereas \textsf{PFA} implies that 
all Aronszajn trees are saturated (\cite{coxkrueger2}, \cite{moorebounding}).

\bigskip

We assume that the reader has a background in $\omega_1$-trees and forcing. 
We refer the reader to \cite[Section 1]{KS} for basic definitions and terminology 
concerning trees. 
An \emph{$\omega_1$-tree} is a tree with height $\omega_1$ and countable levels. 
An \emph{Aronszajn tree} is an $\omega_1$-tree with no cofinal branch. 
A \emph{Kurepa tree} is an $\omega_1$-tree with at least $\omega_2$-many 
cofinal branches, and a \emph{weak Kurepa tree} is a tree with height 
and size $\omega_1$ which has at least $\omega_2$-many cofinal branches. 
The \emph{Kurepa hypothesis} is the statement that there exists a Kurepa tree, 
and the \emph{weak Kurepa hypothesis} is the statement that there exists 
a weak Kurepa tree.

We make use of the following fact about Aronszajn trees, 
which follows from a classic theorem of Baumgartner-Malitz-Reinhardt \cite{baumgartner}
together with the Dushnik-Miller theorem $\omega_1 \to (\omega_1,\omega)^2$. 
Suppose that $T$ is a tree with no uncountable chains and 
$\langle x_\alpha : \alpha < \omega_1 \rangle$ is a sequence of 
disjoint finite subsets of $T$. 
Then there exists a countably infinite set $Y \subseteq \omega_1$ such that for all 
distinct $\alpha$ and $\beta$ in $Y$, for all $x \in x_\alpha$ and for all 
$y \in x_\beta$, $x$ and $y$ are incomparable in $T$.

\section{Finite Trees and Subtrees}

For the remainder of the article, fix a regular cardinal $\kappa \ge \omega_2$. 
In Section 4, we assume that $\kappa$ is equal to $\omega_2$, 
and in Section 5 and for the rest of the 
article we assume that $\kappa$ is an inaccessible cardinal. 
In Sections 3 and 6 we define two forcing posets for adding a strongly non-saturated 
Aronszajn tree. 
The forcing introduced in Section 3 is c.c.c., and the forcing introduced 
in Section 6 is proper and involves countable models as side conditions. 
These two forcing posets will have a part in common which adds the Aronszajn tree. 
In this section, we work out the details of this common part.

Define $h : \omega_1 \to \omega_1$ by letting $h(\alpha)$ be the unique 
ordinal $\gamma$ such that $\omega \cdot \gamma \le \alpha < \omega \cdot (\gamma+1)$. 
Let $C_h$ be the club set of ordinals in $\omega_1$ which are closed under the function 
which maps any $\gamma < \omega_1$ to $\omega \cdot (\gamma+1)$. 
Note that if $\delta \in C_h$, then $\alpha < \delta$ iff $h(\alpha) < \delta$. 
The function $h$ will coincide with the height function of the generic trees  
of Sections 3 and 6.

\begin{definition}
	A \emph{standard finite tree} is an ordered pair $(T,<_T)$ satisfying:
	\begin{enumerate}
	\item $T$ is a finite subset of $\{ 0 \} \cup (\omega_1 \setminus \omega)$;
	\item $<_T$ is a strict partial ordering of $T$ such that 
	for any $x \in T$, the set $\{ y \in T : y <_T x \}$ is linearly ordered by $<_T$;
	\item if $x <_T y$, then $h(x) < h(y)$;
	\item if $T$ is non-empty, then $0 \in T$ and 
	$0 <_T x$ for all non-zero $x \in T$.
	\end{enumerate}
	If $(T,<_T)$ and $(U,<_U)$ are standard finite trees, 
	$(U,<_U)$ is an \emph{end-extension} of $(T,<_T)$ if 
	$T \subseteq U$ and $<_U \cap \ T^2 = \ <_T$.
\end{definition}

We will abbreviate a standard finite tree $(T,<_T)$ as just $T$. 
If $T$ is a standard finite tree and $\delta < \omega_1$, let 
$T \res \delta$ be equal to $(T \cap \delta,<_T \cap \ \delta^2)$. 
Note that $T \res \delta$ is also a standard finite tree. 
Observe that if $T$ is a standard finite tree, then for any incomparable  
elements $x$ and $y$ of $T$, there exists a $<_T$-largest $z \in T$ such that 
$z <_T x, y$. 
We use the notation $h[T]$ for $\{ h(x) : x \in T \}$.

\begin{definition}
	A standard finite tree $T$ is \emph{downwards closed} if whenever $x \in T$ 
	and $\alpha \in h[T] \cap h(x)$, 
	then there exists some $z \in T$ with $z <_T x$ and $h(z) = \alpha$.
\end{definition}

\begin{definition}
	A standard finite tree $T$ has \emph{minimal splits} if 
	whenever $x$ and $y$ are incomparable elements of $T$ and $z$ is the largest 
	element of $T$ below both $x$ and $y$, then 
	there exist distinct $x_0$ and $y_0$ such that 
	$z <_T x_0 \le_T x$, $z <_T y_0 \le_T y$, and $h(x_0) = h(y_0) = h(z) + 1$.
\end{definition}

\begin{definition}
	Let $T$ be a standard finite tree. 
	A \emph{subtree} of $T$ is an ordered pair $(W,<_W)$, where $W \subseteq T$ and 
	$<_W \ = \ <_T \cap \ W^2$.
\end{definition}

Any subset $W$ of $T$ can be considered as a subtree of $T$ 
with the induced tree order $<_W \ = \ <_T \cap \ W^2$. 
We will abbreviate a subtree $(W,<_W)$ as just $W$.

\begin{definition}
	Let $T$ be a standard finite tree and let $W$ be a subtree of $T$. 
	We say that $W$ is \emph{downwards closed in $T$} if 
	whenever $x \in W$ and $y <_T x$, then $y \in W$.
	If $W$ is a subtree of $T$, then the \emph{downward closure of $W$ in $T$} is the 
	set of $y \in T$ such that for some $x \in W$, $y \le_T x$.
\end{definition}

\begin{definition}
	Let $T$ be a standard finite tree. 
	A \emph{subtree function on $T$} is a function $W$ whose domain is a 
	finite subset of $\kappa$ such that for all $\eta \in \dom(W)$, 
	$W(\eta)$ is a downwards closed subtree of $T$.
\end{definition}

If $W$ is a subtree function on $T$ and $\eta \in \kappa$, 
we will occasionally write $W(\eta)$ even if we do not know whether or not $\eta \in \dom(W)$; 
in the case that it is not, $W(\eta)$ should be taken to mean the empty-set.

\begin{definition}
	Let $1 < d < \omega$ and suppose that 
	$T_0,\ldots,T_{d-1}$ are standard finite trees. 
	Define $T_0 \oplus \cdots \oplus T_{d-1}$ to be the ordered pair 
	$$
	(T_0 \cup \cdots \cup T_{d-1},<_{T_0} \cup \cdots \cup <_{T_{d-1}}).
	$$
\end{definition}

\begin{definition}
	Let $1 < d < \omega$. 
	Suppose that for all $i < d$, $T_i$ is a standard finite tree 
	and $W_i$ is a subtree function on $T_i$. 
	Define $W_0 \oplus \cdots \oplus W_{d-1}$ to be the function with domain equal to 
	$\dom(W_0) \cup \cdots \cup \dom(W_{d-1})$ such that for all 
	$\eta \in \dom(W_0 \oplus \cdots \oplus W_{d-1})$, 
	$$
	(W_0 \oplus \cdots \oplus W_{d-1})(\eta) = 
	W_0(\eta) \cup \cdots \cup W_{d-1}(\eta).
	$$
\end{definition}

In general, $T_0 \oplus \cdots \oplus T_{d-1}$ is not necessarily a 
standard finite tree, and even if it is, then 
$W_0 \oplus \cdots \oplus W_{d-1}$ is not necessarily a subtree function on it.

We define an auxiliary forcing $\p^*$ which will assist with our main forcings 
$\p'$ and $\p$ introduced in Sections 3 and 6. 
We will never force with $\p^*$ itself.

\begin{definition}
	Let $\p^*$ be the forcing poset consisting of conditions 
	which are triples $(T,W,D)$ satisfying:
	\begin{enumerate}
	\item $T$ is a standard finite tree;
	\item $W$ is subtree function on $T$;
	\item $D \subseteq [\dom(W)]^2$.
	\end{enumerate}
	Let $(U,Y,E) \le (T,W,D)$ if:
	\begin{enumerate}
	\item[(a)] $U$ end-extends $T$;
	\item[(b)] $\dom(W) \subseteq \dom(Y)$ and 
	for all $\eta \in \dom(W)$, $W(\eta) \subseteq Y(\eta)$;
	\item[(c)] $D \subseteq E$;
	\item[(d)] if $\{ \eta, \xi \} \in D$ and 
	$x$ is in $Y(\eta) \cap Y(\xi)$, then there exists some 
	$z \in W(\eta) \cap W(\xi)$ such that $x \le_U z$.
	\end{enumerate}
\end{definition}

\begin{notation}
	For any $p \in \p^*$, we write $(T_p,W_p,D_p)$ for $p$, and we write 
	$<_p$ for $<_{T_p}$
\end{notation}

\begin{lemma}
	For any $p \in \p^*$, there exists $q \le p$ 
	in $\p^*$ such that $T_q$ is downwards closed and has minimal splits, 
	$D_q = D_p$, $\dom(W_q) = \dom(W_p)$, 
	and for all $\eta \in \dom(W_q)$, $W_q(\eta)$ is the downward closure 
	of $W_p(\eta)$ in $T_q$. 
	Moreover, for all distinct $\eta, \xi \in \dom(W_p)$, if 
	$x \in W_q(\eta) \cap W_q(\xi)$, then there exists some $z \in W_p(\eta) \cap W_p(\xi)$ 
	such that $x \le_{q} z$.
\end{lemma}

\begin{proof}
	The proof is straightforward and we leave it as an exercise 
	for the interested reader.
\end{proof}

\begin{definition}
	Let $1 < d < \omega$. 
	Let $p_0,\ldots,p_{d-1}$ be in $\p^*$. 
	Define $p_0 \oplus \cdots \oplus p_{d-1}$ to be the triple $(T,W,D)$ satisfying:
	\begin{enumerate}
	\item $T = T_0 \oplus \cdots \oplus T_{d-1}$;
	\item $W = W_0 \oplus \cdots \oplus W_{d-1}$;
	\item $D = D_0 \cup \cdots \cup D_{d-1}$.
\end{enumerate}
\end{definition}

In general, $p_0 \oplus \cdots \oplus p_{d-1}$ is not necessarily 
a condition in $\p^*$, and even if it is, it is not necessarily an extension of 
$p_0,\ldots,p_{d-1}$.

\begin{definition}
	Let $p$ and $q$ be in $\p^*$ and let $\delta_p < \delta_q$ be in $C_h$. 
	We say that the ordered pair $(p,q)$ is \emph{$(\delta_p,\delta_q)$-split} if 
	the following are satisfied:
	\begin{enumerate}
		\item $T_p \res \delta_p = T_q \res \delta_q$;
		\item $T_p \subseteq \delta_q$;
		\item for all $\eta \in \dom(W_p) \cap \dom(W_q)$, 
		$W_p(\eta) \cap \delta_p = W_q(\eta) \cap \delta_q$;
		\item for all distinct $\eta, \xi \in \dom(W_p) \cap \dom(W_q)$, 
		$W_p(\eta) \cap W_p(\xi) \subseteq \delta_p$ and  
		$W_q(\eta) \cap W_q(\xi) \subseteq \delta_q$.
	\end{enumerate}
\end{definition}

\begin{lemma}
	Suppose that $p$ and $q$ are $(\delta_p,\delta_q)$-split, where $p, q \in \p^*$ 
	and $\delta_p < \delta_q$ are in $C_h$. 
	Then:
	\begin{enumerate}
		\item[(a)] $T_p \cap T_q \subseteq \delta_p$;
		\item[(b)] if $\eta \in \dom(W_p)$ and $\xi \in \dom(W_p) \cap \dom(W_q)$, 
		then $W_p(\eta) \cap W_q(\xi) \subseteq W_p(\xi)$;
		\item[(c)] if $\eta \in \dom(W_q)$ and $\xi \in \dom(W_p) \cap \dom(W_q)$, 
		then $W_q(\eta) \cap W_p(\xi) \subseteq W_q(\xi)$.
	\end{enumerate}
\end{lemma}

\begin{proof}
	(a) follows from Definition 2.13(1,2). 
	For (b), assume that $x \in W_p(\eta) \cap W_q(\xi)$, and we show that $x \in W_p(\xi)$. 
	Since $\xi \in \dom(W_p) \cap \dom(W_q)$, 
	by Definition 2.13(3) we have that 
	$W_p(\xi) \cap \delta_p = W_q(\xi) \cap \delta_q$. 
	But $x \in W_p(\eta) \cap W_q(\xi)$ implies that 
	$x \in T_p \cap T_q \subseteq \delta_q$. 
	So $x \in W_q(\xi) \cap \delta_q \subseteq W_p(\xi)$. 	
	For (c), assume that $x \in W_q(\eta) \cap W_p(\xi)$, 
	and we show that $x \in W_q(\xi)$. 
	Since $\xi \in \dom(W_p) \cap \dom(W_q)$, by Definition 2.13(3) we have that  
	$W_p(\xi) \cap \delta_p = W_q(\xi) \cap \delta_q$. 
	But $x \in W_q(\eta) \cap W_p(\xi)$ implies that 
	$x \in T_p \cap T_q \subseteq \delta_p$. 
	So $x \in W_p(\xi) \cap \delta_p \subseteq W_q(\xi)$.
\end{proof}

\begin{lemma}
	Let $1 < d < \omega$. 
	Suppose that $p_0,\ldots,p_{d-1}$ are in $\p^*$ and 
	$\delta_0 < \cdots < \delta_{d-1}$ are in $C_h$. 
	Assume:
	\begin{enumerate}
		\item for all $i < j < d$, $(p_i,p_j)$ is $(\delta_i,\delta_j)$-split;
		\item $\{ \dom(W_i) : i < d \}$ is a $\Delta$-system.
	\end{enumerate}
	Then $p_0 \oplus \cdots \oplus p_{d-1}$ is a condition in $\p^*$ 
	which extends $p_0,\ldots,p_{d-1}$.
\end{lemma}

\begin{proof}
	For each $i < d$, write $p_i = (T_i,W_i,D_i)$, and write 
	$p_0 \oplus \cdots \oplus p_{d-1} = (T,W,D)$. 
	Let $r$ be the root of the $\Delta$-system of (2). 
	We claim that $T$ is a standard finite tree. 
	For transitivity, suppose that $a <_T b <_T c$. 
	If for some $i < d$, $a <_{T_i} b <_{T_i} c$, then we are done 
	since $p_i$ is a condition. 
	Otherwise, for some distinct $i, j < d$, 
	$a <_{T_i} b <_{T_j} < c$. 
	Let $k = \min \{ i, j \}$. 
	Then $b \in T_i \cap T_j \subseteq \delta_k$. 
	As $T_i \res \delta_k = T_j \res \delta_k$, $a <_{T_j} b$, 
	so $a <_{T_j} c$ and we are done.

	Now let $c \in T$ and assume that $a <_T c$ and $b <_T c$, where $a$ and $b$ are different. 
	We will show that $a$ and $b$ are comparable in $T$. 
	If for some $i < d$, $a <_i c$ and $b <_i c$, then we are done since $p_i$ is a condition. 
	Otherwise, for some distinct $i, j < d$, 
	$a <_i c$ and $b <_j c$. 
	Let $k = \min \{ i, j \}$. 
	Then $c \in T_i \cap T_j \subseteq \delta_k$, 
	and since $T_i \res \delta_k = T_j \res \delta_k$, 
	it follows that $a <_{T_i} c$ and $b <_{T_i} c$. 
	Thus, $a$ and $b$ are comparable 
	in $T_i$ and therefore also in $T$.
	The other properties of $T$ being a standard finite tree are obvious.

	Since $\dom(W_i) \subseteq \dom(W)$ are for all $i < d$, 
	obviously $D \subseteq [\dom(W)]^2$. 
	We claim that for all $\eta \in \dom(W)$, $W(\eta)$ is a downwards closed subtree of $T$. 
	Let $y \in W(\eta)$ and suppose that $x <_T y$. 
	Fix $i, j < d$ such that $y \in W_i(\eta)$ and $x <_{T_j} y$. 
	If $i = j$, then we are done since $p_i$ is a condition. 
	Assume that $i \ne j$. 
	Let $k = \min \{ i, j \}$. 
	Then $y \in T_i \cap T_j \subseteq \delta_k$. 
	As $T_i \res \delta_k = T_j \res \delta_k$, $x <_{T_i} y$. 
	Since $y \in W_i(\eta)$ and $W_i(\eta)$ is downwards closed in $T_i$, 
	$x \in W_i(\eta)$.

	This completes the proof that $(T,W,D)$ is in $\p^*$. 
	It remains to show that for all $i < d$, $(T,W,D)$ is an extension of $(T_i,W_i,D_i)$. 
	Clearly, $T_{i} \subseteq T$ and $<_{T_i} \ \subseteq \ <_T \cap \ T_{i}^2$. 
	To show that $T$ is an end-extension of $T_i$, 
	suppose that $a, b \in T_i$ and $a <_{T} b$. 
	If $a <_i b$, then we are done. 
	Otherwise, for some $j < d$ different from $i$, $a <_j b$. 
	Let $k = \min \{ i, j \}$. 
	Then $a, b \in T_i \cap T_j \subseteq \delta_k$, and since 
	$T_i \res \delta_k = T_j \res \delta_k$, $a <_i b$. 
	This proves that $T$ end-extends $T_i$. 
	It is clear that $\dom(W_i) \subseteq \dom(W)$ 
	and for all $\eta \in \dom(W_i)$, 
	$W_i(\eta) \subseteq W(\eta)$. 
	And obviously $D_i \subseteq D$.

	Finally, let $\{ \eta, \xi \} \in D_i$ and let $x \in W(\eta) \cap W(\xi)$. 
	We find some $z \in W_i(\eta) \cap W_i(\xi)$ such that $x \le_T z$. 
	In fact, we prove that $x \in W_i(\eta) \cap W_i(\xi)$, so $x = z$ works. 
	Note that $\eta, \xi \in \dom(W_i)$. 
	Fix $l, m < d$ such that $x \in W_l(\eta) \cap W_m(\xi)$. 
	Assume first that one of $l$ or $m$ is equal to $i$. 
	Without loss of generality, assume that $i = l$. 
	Then $x \in W_i(\eta)$. 
	If $i = m$, then $x \in W_i(\eta) \cap W_i(\xi)$ and we are done, 
	so assume that $i \ne m$. 
	Then $\eta \in \dom(W_i)$, $\xi \in \dom(W_i) \cap \dom(W_m)$, 
	and $x \in W_i(\eta) \cap W_m(\xi)$. 
	It follows by Lemma 2.14(b,c) that $x \in W_i(\xi)$ and we are done.

	Now assume that both $l$ and $m$ are not equal to $i$. 	
	Then $\xi \in \dom(W_i) \cap \dom(W_m) = r \subseteq \dom(W_l)$. 
	So $\eta \in \dom(W_l)$, $\xi \in \dom(W_l) \cap \dom(W_m)$, 
	and $x \in W_l(\eta) \cap W_m(\xi)$. 
	By Lemma 2.13(b,c), $x \in W_l(\xi)$. 
	So in fact $x \in W_l(\eta) \cap W_l(\xi)$. 
	We have that $\eta$ and $\xi$ are both in $\dom(W_i) \cap \dom(W_l)$. 
	So by Definition 2.13(4), $W_l(\eta) \cap W_l(\xi) \subseteq \delta_l$. 
	By Definition 2.13(3), 
	$W_l(\eta) \cap \delta_l = W_i(\eta) \cap \delta_i$ and 
	$W_l(\xi) \cap \delta_l = W_i(\xi) \cap \delta_i$. 
	So $x \in W_i(\eta) \cap W_i(\xi)$ and we are done.
\end{proof}

\section{The First Forcing}

In this section, we define a c.c.c.\ forcing for adding a strongly non-saturated 
Aronszajn tree assuming the existence of a function from $\kappa^2$ into $\omega_1$ 
with a special property. 
This forcing is a suborder of $\p^*$ consisting of conditions which satisfy a 
restriction on the intersection of their subtrees related to this function.

For the remainder of this section, 
fix a function $e : \kappa^2 \to \omega_1$.

\begin{definition}[$e$-Separation]
	Let $T$ be a standard finite tree and let $W$ be a subtree function on $T$. 
	We say that $W$ is \emph{$e$-separated} if for all distinct $\eta, \xi \in \dom(W)$, 
	if $x \in W(\eta) \cap W(\xi)$, then $e(\eta,\xi) \ge h(x)$.
\end{definition}

\begin{definition}
	Let $\p'$ be the suborder of $\p^*$ consisting of triples 
	$(T,W,D) \in \p^*$ such that $W$ is $e$-separated.
\end{definition}

We begin by analyzing the generic object which is added by $\p'$.

\begin{definition}
	For any generic filter $G$ on $\p'$, define $(T_G,<_G)$ by:
	\begin{itemize}
		\item $x \in T_G$ if there exists some $p \in G$ such that $x \in T_p$;
		\item $x <_G y$ if there exists some $p \in G$ such that $x <_p y$.
	\end{itemize}
\end{definition}

We abbreviate $(T_G,<_G)$ by $T_G$. 
We occasionally write $\dot G$ for the canonical $\p'$-name for a 
generic filter on $\p'$.
Let $T_{\dot G}$ be a $\p'$-name for the above object.

The following is easy to verify.

\begin{lemma}
	If $G$ is a generic filter on $\p'$, then $T_G$ is a tree with a root.
\end{lemma}

\begin{lemma}
	For any $p \in \p'$, there exists $q \le p$ such that 
	$T_q$ is downwards closed and has minimal splits.
\end{lemma}

\begin{proof}
	Fix $q \le p$ in $\p^*$ 
	satisfying the properties described in Lemma 2.11. 
	In particular, $T_q$ is downwards closed and has minimal splits. 
	To show that $q$ is in $\p'$, we prove that $W_q$ is $e$-separated. 
	So let $\eta$ and $\xi$ be distinct elements of $\dom(W_q)$, and assume 
	that $x \in W_q(\eta) \cap W_q(\xi)$. 
	By Lemma 2.11, there exists some $z \in W_p(\eta) \cap W_p(\xi)$ 
	such that $x \le_q z$. 
	Since $W_p$ is $e$-separated, $e(\eta,\xi) \ge h(z) \ge h(x)$.
\end{proof}

The next lemma is easy using Lemma 3.5.

\begin{lemma}
	Let $p \in \p'$. 
	\begin{enumerate}
		\item If $x \in T_p$ and $\alpha < h(x)$, then there exists 
		$q \le p$ and $y \in T_q$ such that $h(y) = \alpha$ and $y <_q x$.
		\item If $x \in T_p$ and $h(x) < \beta < \omega_1$, then 
		there exists $q \le p$ and $y \in T_q \setminus T_p$ 
		such that $h(y) = \beta$ and $x <_U y$.
	\end{enumerate}
\end{lemma}

\begin{lemma}
	Let $G$ be a generic filter on $\p'$. 
	Then the height function of $T_G$ coincides with $h$. 
	So $T_G$ has height $\omega_1^V$.
\end{lemma}

\begin{proof}
	By Lemma 3.6(1).
\end{proof}

\begin{definition}
	For any generic filter $G$ on $\p'$ and for any $\eta < \kappa$, define 
	$$
	W_G(\eta) = \bigcup \{ W_p(\eta) : p \in G, \ \eta \in \dom(W_p) \}.
	$$
\end{definition}

The next two lemmas are easy to check.

\begin{lemma}
	For any $p \in \p'$, for any $\eta < \kappa$, and for any $\alpha < \omega_1$, 
	there exists $q \le p$ such that $\eta \in \dom(W_q)$ and there exists 
	some $x \in W_q(\eta)$ with $h(x) = \alpha$.
\end{lemma}

\begin{lemma}
	For any generic filter $G$ on $\p'$ and for any $\eta < \kappa$, 
	$W_G(\eta)$ is an uncountable downwards closed subtree of $T_G$.
\end{lemma}

\begin{lemma}
	Let $G$ be a generic filter on $\p'$ and let $\eta < \xi < \kappa$. 
	Then $W_G(\eta) \cap W_G(\xi)$ is finitely generated, and hence 
	$W_G(\eta)$ and $W_G(\xi)$ are strongly almost disjoint.
\end{lemma}

\begin{proof}
	It is easy to prove that for all $p \in \p'$, 
	there exists $q \le p$ such that $\{ \eta, \xi \} \in D_q$. 
	It follows that there exists some $q \in G$ such that $\{ \eta, \xi \} \in D_q$. 
	We claim that any member of $W_G(\eta) \cap W_G(\xi)$ is less than or equal to 
	some member of the finite set $W_q(\eta) \cap W_q(\xi)$ in $T_G$. 
	So let $x \in W_G(\eta) \cap W_G(\xi)$. 
	Then there exists some 
	$r \le q$ in $G$ such that $x \in W_r(\eta) \cap W_r(\xi)$. 
	Since $\{ \eta, \xi \} \in D_q$, by the definition of $\p'$ 
	there exists some $z \in W_q(\eta) \cap W_q(\xi)$ such that $x \le_r z$. 
	Then $x \le_{G} z$.
\end{proof}

We now describe a special property of $e$ and prove that it implies that 
$\p'$ is c.c.c.

\begin{definition}
	A function $f : \kappa^2 \to \omega_1$ 
	is a \emph{weak $\rho$-function} if whenever 
	$\langle F_i : i < \omega_1 \rangle$ is a pairwise disjoint 
	sequence of finite subsets of $\kappa$, then 
	for any $\gamma < \omega_1$ there exist $i < j < \omega_1$ such that 
	for all $\eta \in F_i$ and for all $\xi \in F_j$, 
	$f(\eta,\xi) \ge \gamma$. 
\end{definition}

\begin{thm}
	Suppose that $e$ is a weak $\rho$-function. 
	Assume that $\langle p_\alpha : \alpha < \omega_1 \rangle$ 
	is a sequence of conditions in $\p'$. 
	Then for some $\alpha < \beta$ in $C_h$, 
	$(p_\alpha,p_\beta)$ is $(\alpha,\beta)$-split 
	and $p_\alpha \oplus p_\beta$ 
	is a condition in $\p'$ which extends $p_\alpha$ and $p_\beta$.
\end{thm}

\begin{proof}
	Write $p_\alpha = (T^\alpha,W^\alpha,D^\alpha)$ for all $\alpha < \omega_1$. 
	For each $\alpha < \omega_1$, enumerate $\dom(W^\alpha)$ in increasing order as 
	$\langle \eta_0^\alpha, \ldots, \eta_{n_\alpha-1}^\alpha \rangle$. 

	By a standard pressing down argument, we can find a stationary set 
	$Z_0 \subseteq C_h \cap \cof(> \! \omega)$, 
	a standard finite tree $T$, $n < \omega$, 
	and downwards closed subtrees $w_0,\ldots,w_{n-1}$ of $T$ 
	such that for all $\alpha \in Z_0$:
	\begin{itemize}
		\item $T^\alpha \res \alpha = T$;
		\item $n_\alpha = n$;
		\item for all $k < n$, $W^\alpha(\eta_k^\alpha) \cap \alpha = w_k$;
	\end{itemize}
	and moreover, for all $\alpha < \beta$ in $Z_0$, $T^\alpha \subseteq \beta$.
	
	Applying the $\Delta$-system lemma, 
	fix an uncountable set $Z_1 \subseteq Z_0$ 
	and a finite set $r \subseteq \kappa$ 
	such that for all 
	$\alpha < \beta$ in $Z_1$, $\dom(W^\alpha) \cap \dom(W^\beta) = r$. 
	Now find an uncountable set $Z \subseteq Z_1$, an ordinal $\zeta \in C_h$, 
	and a set $x \subseteq n$ 
	such that for all $\alpha \in Z$:
	\begin{itemize}
		\item $\{ k < n : \eta_k^\alpha \in r \} = x$ 
		\item $T \subseteq \zeta$;
		\item $\{ e(\eta,\xi) : \eta, \xi \in r \} \subseteq \zeta$.
	\end{itemize}
	Note that for all $\alpha < \beta$ in $Z$ and for all $k \in x$, 
	$\eta_k^\alpha = \eta_k^\beta$.
	
	\textbf{Claim:} For all $\alpha < \beta$ in $Z$, $(p_\alpha,p_\beta)$ 
	is $(\alpha,\beta)$-split.
	
	\emph{Proof:} (1) and (2) of Definition 2.13 are immediate by the choice of 
	$T$ and $Z$. 
	For (3), let $\eta \in \dom(W^\alpha) \cap \dom(W^\beta) = r$. 
	Then for some $k \in x$, $\eta = \eta_k^\alpha = \eta_k^\beta$. 
	Hence, $W^\alpha(\eta) \cap \alpha = W^\alpha(\eta_k^\alpha) \cap \alpha = 
	w_k = W^\beta(\eta_k^\beta) \cap \beta = W^\beta(\eta) \cap \beta$. 
	For (4), let $\eta$ and $\xi$ be distinct elements of 
	$\dom(W^\alpha) \cap \dom(W^\beta)$. 
	Then $\eta, \xi \in r$. 
	So $e(\eta,\xi) < \zeta < \alpha$. 
	Let $x \in W^\alpha(\eta) \cap W^\alpha(\xi)$ and we show that $x < \alpha$. 
	Since $W^\alpha$ is $e$-separated, 
	$h(x) \le e(\eta,\xi) < \alpha$, so $x < \alpha$. 
	A similar argument show that 
	$W^\beta(\eta) \cap W^\beta(\xi) \subseteq \beta$. 
	This completes the proof of the claim.
	
	By Lemma 2.15, it follows that for all $\alpha < \beta$ in $Z$, 
	$p_\alpha \oplus p_\beta$ is in $\p^*$ and is an extension of $p_\alpha$ and $p_\beta$. 
	Applying the assumption that $e$ is a weak $\rho$-function to 
	$\langle \dom(W^\alpha) \setminus r : \alpha \in Z \rangle$, 
	fix $\alpha < \beta$ in $Z$ such that for all 
	$\eta \in \dom(W^\alpha) \setminus r$ 
	and for all $\xi \in \dom(W^\beta) \setminus r$, 
	$e(\eta,\xi) \ge \zeta$. 

	We claim that $p_\alpha$ and $p_\beta$ are as required. 
	We already know that $(p_\alpha,p_\beta)$ is $(\alpha,\beta)$-split 
	and $p_\alpha \oplus p_\beta$ is in $\p^*$ and is an 
	extension of $p_\alpha$ and $p_\beta$. 
	So it suffices to show that $W^\alpha \oplus W^\beta$ is $e$-separated. 
	Consider distinct $\eta$ and $\xi$ 
	in $\dom(W^\alpha \oplus W^\beta)$ and 
	assume that $x \in (W^\alpha \oplus W^\beta)(\eta) \cap 
	(W^\alpha \oplus W^\beta)(\xi)$. 
	We show that $e(\eta,\xi) \ge h(x)$.

	\emph{Case 1:} $x \in T^\alpha \setminus \alpha$. 
	Then $x \notin T^\beta$. 
	By the definition of $W^\alpha \oplus W^\beta$, we must have that 
	$x \in W^\alpha(\eta) \cap W^\alpha(\xi)$, for the other 
	other possibilities imply that $x \in T^\beta$. 
	Since $W^\alpha$ is $e$-separated, 
	it follows that $e(\eta,\xi) \ge h(x)$.

	\emph{Case 2:} $x \in T^\beta \setminus \beta$. 
	Then $x \notin T^\alpha$. 
	So as in Case 1, $x \in W^\beta(\eta) \cap W^\beta(\xi)$. 
	Since $W^\beta$ is $e$-separated, 
	it follows that $e(\eta,\xi) \ge h(x)$.

	\emph{Case 3:} $x \in T$. 
	Then $h(x) < \zeta$.  
	If one of $\eta$ or $\xi$ is in 
	$\dom(W^\alpha) \setminus r$ and the other is in $\dom(W^\beta) \setminus r$, 
	then by the choice of $\alpha$ and $\beta$ we have that 
	$e(\eta,\xi) \ge \zeta > h(x)$ and we are done. 
	Otherwise, one of $\eta$ or $\xi$ is in $\dom(W^\alpha) \cap \dom(W^\beta)$. 
	Without loss of generality, assume that $\xi \in \dom(W^\alpha) \cap \dom(W^\beta)$. 
	If $x$ is either in $W^\alpha(\eta) \cap W^\alpha(\xi)$ or in 
	$W^\beta(\eta) \cap W^\beta(\xi)$, then $e(\eta,\xi) \ge h(x)$ 
	by the fact that $W^\alpha$ and $W^\beta$ are $e$-separated. 
	So assume not. 
	Then either $x \in W^\alpha(\eta) \cap W^\beta(\xi)$ or 
	$x \in W^\beta(\eta) \cap W^\alpha(\xi)$. 
	Since $(p_\alpha,p_\beta)$ is $(\alpha,\beta)$-split, 
	by Lemma 2.14 we have that $x \in W^\alpha(\xi)$ in the first case and 
	$x \in W^\beta(\xi)$ in the second case. 
	So $x \in W^\alpha(\eta) \cap W^\alpha(\xi)$ in the first case and 
	$x \in W^\beta(\eta) \cap W^\beta(\xi)$ in the second case, both of 
	which contradict our current assumptions.
\end{proof}

\begin{corollary}
	If $e$ is a weak $\rho$-function, then $\p'$ is c.c.c.
\end{corollary}

\begin{corollary}
	If $e$ is a weak $\rho$-function, then $\p'$ forces that 
	$T_{\dot G}$ has no uncountable chain.
\end{corollary}

\begin{proof}
	Suppose for a contradiction that $p \in \p'$ forces that $\dot b$ is an uncountable 
	chain of $T_{\dot G}$. 
	For each $\alpha < \omega_1$, fix a condition $p_\alpha \le p$ and 
	some $x_\alpha \in T_{p_\alpha}$ 
	such that $p_\alpha$ forces that $x_\alpha \in \dot b \setminus \alpha$. 
	Applying Theorem 3.13, find $\alpha < \beta$ in $C_h$ such that 
	$(p_\alpha,p_\beta)$ is $(\alpha,\beta)$-split and 
	$p_\alpha \oplus p_\beta$ 
	is a condition in $\p'$ which extends both $p_\alpha$ and $p_\beta$. 
	Write $T_{p_\alpha} = T^\alpha$ and $T_{p_\beta} = T^\beta$.
	
	Since $(p_\alpha,p_\beta)$ is $(\alpha,\beta)$-split, 
	$T^\alpha \res \alpha = T^\beta \res \beta$ and $T^\alpha \subseteq \beta$. 
	As $x_\alpha \ge \alpha$, $x_\alpha \in T^\alpha \setminus T^\beta$, and since 
	$x_\beta \ge \beta$, $x_\beta \in T^\beta \setminus T^\alpha$. 
	By the definition of $T^\alpha \oplus T^\beta$, 
	$x_\alpha$ and $x_\beta$ are incomparable in $T^\alpha \oplus T^\beta$. 
	Now for any $r \le p_\alpha \oplus p_\beta$, $T_r$ is an end-extension of 
	$T^\alpha \oplus T^\beta$, and therefore $x_\alpha$ and $x_\beta$ 
	are incomparable in $T_r$. 
	Consequently, $p_\alpha \oplus p_\beta$ forces that $x_\alpha$ and $x_\beta$ 
	are incomparable in $T_{\dot G}$, which contradicts that $p_\alpha \oplus p_\beta$ 
	forces that $x_\alpha$ and $x_\beta$ are both in the chain $\dot b$.
\end{proof}

\begin{thm}
	Suppose that $e$ is a weak $\rho$-function. 
	Let $G$ be a generic filter on $\p'$. 
	Then $T_G$ is a normal infinitely splitting Aronszajn tree 
	and $\{ W_G(\eta) : \eta < \kappa \}$ is a pairwise strongly almost disjoint 
	family of uncountable downwards closed subtrees of $T_G$ witnessing that 
	$T_G$ is strongly non-saturated.
\end{thm}

\begin{proof}
	Lemma 3.7 and Corollary 3.15 imply that $T_G$ is an Aronszajn tree. 
	Lemma 3.5 implies that $T_G$ is Hausdorff, and Lemma 3.6(2) implies that 
	$T_G$ is normal and infinitely splitting. 
	By Lemmas 3.10 and 3.11 we are done. 
\end{proof}

\section{The Main Theorems: Part 1}

Our first main theorem is proven by combining the results of the previous section 
with work of Jensen-Schlechta and Todor\v{c}evi\'{c}. 
Recall that the \emph{generic Kurepa hypothesis} (\textsf{GKH}) 
is the statement that there 
exists a Kurepa tree in some c.c.c.\ forcing extension (\cite{jensenschlechta}). 
So $\neg \textsf{GKH}$ is equivalent to 
the statement that the negation of Kurepa's hypothesis is c.c.c.\ indestructible.

\begin{thm}[{\cite[Lemma 4]{todorchang}}]
	The negation of Chang's conjecture is equivalent to the existence of a 
	weak $\rho$-function $e : (\omega_2)^2 \to \omega_1$.
\end{thm}

\begin{thm}[{\cite[Proposition 1.4]{jensenschlechta}}]
	Suppose that $\kappa$ is a Mahlo cardinal. 
	Then the L\'{e}vy collapse $\col(\omega_1,<\kappa)$ forces 
	$\neg \textsf{GKH}$.
\end{thm}

Finally, we use the fact that Chang's conjecture implies the existence 
of $0^{\#}$ and therefore fails in any generic extension of $L$.

\begin{thm}
	Suppose that there exists a Mahlo cardinal. 
	Then there is a generic extension of $L$ in which there exists 
	a strongly non-saturated Aronszajn tree and $\neg \textsf{GKH}$ holds.
\end{thm}

\begin{proof}
	Let $\kappa$ be a Mahlo cardinal. 
	Then $\kappa$ is a Mahlo cardinal in $L$. 
	Let $K$ be an $L$-generic filter on the L\'{e}vy collapse 
	$\col(\omega_1,< \! \kappa)^L$. 
	Working in $L[K]$, $0^{\#}$ does not exist and hence 
	Chang's conjecture fails. 
	By Theorem 4.1, in $L[K]$ we can fix a 
	weak $\rho$-function $e : (\omega_2)^2 \to \omega_1$. 
	Define $\p'$ in $L[K]$ using $\kappa = \omega_2$ and the function $e$. 

	Let $G$ be an $L[K]$-generic filter on $\p'$. 
	By Theorem 3.16, in $L[K][G]$ we have that $T_G$ is a normal 
	infinitely splitting Aronszajn tree which is strongly non-saturated. 
	Consider any c.c.c.\ forcing poset $\q$ in $L[K][G]$. 
	In $L[K]$ fix a $\p'$-name $\dot \q$ for a c.c.c.\ forcing such that $\dot \q^G = \q$. 
	Then the two-step forcing iteration $\p' * \dot \q$ is c.c.c.\ in $L[K]$, 
	and hence by Theorem 4.2, 
	$\p' * \dot \q$ forces over $L[K]$ that there does not exist a Kurepa tree. 
	So in $L[K][G]$, $\q$ forces that there does not exist a Kurepa tree. 
	So $\neg \textsf{GKH}$ holds in $L[K][G]$.
\end{proof}

We note that in contrast to the model of \cite{KS}, in the model of the above theorem 
$\neg \textsf{GKH}$ implies that there does not exist an almost Kurepa Suslin tree.

\section{Adequate Sets}

We now turn to developing our second forcing poset for adding a 
strongly non-saturated Aronszajn tree. 
For the remainder of the article, assume that $\kappa$ is an inaccessible cardinal. 
In this section we review the type of side conditions which are used in this forcing. 
We refer the reader to \cite{JK21} for the proofs of Proposition 5.8 and 
Theorems 5.11 and 5.15 below, as well as for a general discussion of this style of 
side conditions and its history. 
Other than these three black boxes, 
we include the remaining proofs for completeness, all of which are easy. 
We do note that in \cite{JK21} the context is a bit different, since 
the inaccessible cardinal $\kappa$ is replaced with $\omega_2$. 
But everything works almost identically in both cases, with the difference 
being that in our current situation $\kappa$ will be collapsed to become $\omega_2$.

Fix a bijection $\psi : \kappa \to H(\kappa)$. 
Define a well-ordering $\lhd$ of $H(\kappa)$ by $a \lhd b$ if 
$\psi^{-1}(a) < \psi^{-1}(b)$.  
Let $\mathcal A$ denote the structure $(H(\kappa),\in,\psi)$. 
Since $\lhd$ is a well-ordering of $H(\kappa)$ which is definable in $\mathcal A$, 
the structure $\mathcal A$ has definable Skolem functions. 
For any set $x \subseteq H(\kappa)$, let $\sk(x)$ denote the closure of 
$x$ under these definable Skolem functions. 
And let $\text{cl}(x)$ denote the set consisting of the elements of $x$ 
together with the limit points of $x$. 
Define $\Lambda_0$ to be the club of all $\beta < \kappa$ such that 
$\sk(\beta) \cap \kappa = \beta$.

\begin{definition}
	Define $\Lambda$ to be the set of all $\beta < \kappa$ with uncountable cofinality 
	which are limit points of $\Lambda_0$ and satisfy that 
	$[\beta]^\omega \subseteq \sk(\beta)$.
\end{definition}

Since $\kappa$ is inaccessible, $\Lambda$ is the intersection of some club subset 
of $\kappa$ with $\kappa \cap \cof(> \! \omega)$.

In \cite{JK21} we also fix a thin stationary subset of $[\omega_2]^\omega$ 
which is only needed in the case that \textsf{CH} is false. 
In this article, this set will just be $[\kappa]^\omega$ and will not be mentioned explicitly.

\begin{definition}
	Define $\mathcal X$ to be the set of all $N \in [\kappa]^\omega$ such that 
	$\sk(N) \cap \kappa = N$ and for all $\gamma \in N$, 
	$\sup(\gamma \cap \Lambda_0) \in N$.
\end{definition}

The main point of this definition for us is the property that $\sk(N) \cap \kappa = N$. 
The second requirement is of minor technical importance 
and can be ignored for this article.

Note that $\mathcal X$ is a club subset of $[\kappa]^\omega$. 
It is easy to check that 
$\mathcal X$ is closed under intersections and if $M \in \mathcal X$ 
and $\beta \in \Lambda$, then $M \cap \beta \in \mathcal X$. 
Observe that if $\beta \in \Lambda$ then $\mathcal X \cap \mathcal P(\beta) = 
\mathcal X \cap \sk(\beta)$.

\begin{definition}
For all $M, N \in \mathcal X$, define 
$$
\beta_{\scaleto{M,N}{5pt}} = 
\min(\Lambda \setminus \sup(\cl(M) \cap \cl(N))).
$$
The ordinal $\beta_{\scaleto{M,N}{5pt}}$ is called the 
\emph{comparison point of $M$ and $N$}.
\end{definition}

\begin{definition}
	A set $A$ is \emph{adequate} if $A$ is a finite subset of $\mathcal X$ and 
	for all $M$ and $N$ in $A$, one of the following holds:
	\begin{enumerate}
		\item ($M < N$) \ $M \cap \beta_{\scaleto{M,N}{5pt}} \in \sk(N)$;
		\item ($N < M$) \  $N \cap \beta_{\scaleto{M,N}{5pt}} \in \sk(M)$;
		\item ($M \sim N$) \  
		$M \cap \beta_{\scaleto{M,N}{5pt}} = N \cap \beta_{\scaleto{M,N}{5pt}}$.
	\end{enumerate}
\end{definition}

If $A$ is adequate and $M, N \in A$, then $M < N$ is equivalent to 
$M \cap \omega_1 < N \cap \omega_1$, and $M \sim N$ is equivalent to 
$M \cap \omega_1 = N \cap \omega_1$. 
We also write $M \le N$ to mean that either $M < N$ or $M \sim N$.

The next lemma follows easily from the definitions.

\begin{lemma}
	Assume that $A$ is adequate, $M, N \in A$, and $M < N$. 
	Then $M \cap N = M \cap \beta_{M,N}$, and hence $M \cap N \in \sk(N)$.
\end{lemma}

\begin{lemma}
	Suppose that $A$ is adequate, $N \in \mathcal X$, and 
	$A \subseteq \sk(N)$. 
	Then $A \cup \{ N \}$ is adequate.
\end{lemma}

\begin{proof}
	If $M, N \in \mathcal X$ and $M \in \sk(N)$, 
	then easily $\beta_{\scaleto{M,N}{5pt}} > \sup(M)$, and hence 
	$M \cap \beta_{M,N} = M \in \sk(N)$.
\end{proof}

\begin{lemma}
	Suppose that $\chi \ge \kappa$ is regular and $M$ is a countable elementary 
	substructure of $(H(\chi),\in,\psi)$ 
	such that $N = M \cap \kappa \in \mathcal X$. 
	Then $M \cap H(\kappa) = \sk(N)$.
\end{lemma}

\begin{proof}
	Since $M \cap H(\kappa)$ is easily an elementary substructure of 
	$(H(\kappa),\in,\psi)$, $\sk(N) \subseteq M \cap H(\kappa)$. 
	But since $\psi : \kappa \to H(\kappa)$ is a bijection, by elementarity 
	$M \cap H(\kappa) = \psi[N] \subseteq \sk(N)$.
\end{proof}

\begin{proposition}[{\cite[Proposition 3.4]{JK21}}]
	Suppose that $A, C \subseteq \mathcal X$ are finite, $A$ is adequate, 
	$A \subseteq C$, and for all $K \in C \setminus A$, there exists 
	some $M \in A$ and some $\beta \in \Lambda$ such that $K = M \cap \beta$. 
	Then $C$ is adequate.
\end{proposition}

\begin{definition}
	Let $A$ be adequate and let $N \in A$. 
	We say that $A$ is \emph{$N$-closed} if for all $M \in A$, 
	if $M < N$ then $M \cap N \in A$.
\end{definition}

\begin{lemma}
	Suppose that $A$ is adequate and $N \in A$. 
	Then 
	$$
	A \cup \{ M \cap N : M \in A, \ M < N \}
	$$
	is adequate and $N$-closed.
\end{lemma}

\begin{proof}
	This follows immediately from Lemma 5.5 and Proposition 5.8.
\end{proof}

\begin{thm}[{\cite[Proposition 3.9]{JK21}}]
	Let $A$ be adequate, let $N \in A$, and suppose that $A$ is $N$-closed. 
	Assume that $B$ is adequate and 
	$$
	A \cap \sk(N) \subseteq B \subseteq \sk(N).
	$$
	Then $A \cup B$ is adequate.
\end{thm}

In this article, we need an extension of Theorem 5.11 to finitely many adequate sets.

\begin{corollary}
	Let $1 < d < \omega$. 
	Suppose:
	\begin{enumerate}
		\item $A_0,\ldots,A_{d-1}$ are adequate;
		\item for all $0 < i < d$, $N_i \in A_i$ and $A_i$ is $N_i$-closed;
		\item for all $0 < i < d$, 
		$$
		A_i \cap \sk(N_i) \subseteq A_{i-1} \subseteq \sk(N_{i}).
		$$
	\end{enumerate}
	Then $A_0 \cup \cdots \cup A_{d-1}$ is adequate.
\end{corollary}

\begin{proof}
	By induction on $d$ using Theorem 5.11.
\end{proof}

\begin{definition}
	Let $A$ be adequate and let $\beta \in \Lambda$. 
	We say that $A$ is \emph{$\beta$-closed} if for all $M \in A$, 
	$M \cap \beta \in A$.
\end{definition}

\begin{lemma}
	Suppose that $A$ is adequate and $\beta \in \Lambda$. 
	Then 
	$$
	C = A \cup \{ M \cap \beta : M \in A \}
	$$
	is adequate and $\beta$-closed. 
	Moreover, if $N \in A$ and $A$ is $N$-closed, then 
	$C$ is also $N$-closed.
\end{lemma}

\begin{proof}
	The set $C$ is adequate by Proposition 5.8, 
	and it is easily $\beta$-closed. 
	Consider $M \in A$ and we show that 
	$(M \cap \beta) \cap N \in C$. 
	But $(M \cap \beta) \cap N = (M \cap N) \cap \beta$, and since 
	$A$ is $N$-closed, $(M \cap N) \in A$. 
	Hence, $(M \cap N) \cap \beta \in C$.
\end{proof}

\begin{thm}[{\cite[Proposition 3.11]{JK21}}]
	Let $A$ be adequate, let $\beta \in \Lambda$, and assume that $A$ is $\beta$-closed. 
	Suppose that $B$ is adequate and 
	$$
	A \cap \sk(\beta) \subseteq B \subseteq \sk(\beta).
	$$
	Then $A \cup B$ is adequate.
\end{thm}

\begin{lemma}
	Suppose that $\beta \in \Lambda$, 
	$A \subseteq \sk(\beta)$ is adequate, 
	$N \in \mathcal X$, and $N \cap \beta \in A$. 
	Then $A \cup \{ N \}$ is adequate.
\end{lemma}

\begin{proof}
	Note that for all $M \in \mathcal X \cap \sk(\beta)$, 
	$\cl(M) \cap \cl(N) = \cl(M) \cap \cl(N \cap \beta)$, and therefore by definition 
	$\beta_{\scaleto{M,N}{5pt}} = \beta_{\scaleto{M,N \cap \beta}{5pt}}$. 
	Since $\beta \in \Lambda$, $\beta_{\scaleto{M,N}{5pt}} \le \beta$. 
	Hence, 
	$N \cap \beta_{\scaleto{M,N}{5pt}} = (N \cap \beta) \cap \beta_{\scaleto{M,N}{5pt}}$. 
	Let $M \in A$. 
	If $N \cap \omega_1 < M \cap \omega_1$, then $(N \cap \beta) < M$, so 
	$N \cap \beta_{\scaleto{M,N}{5pt}} = 
	(N \cap \beta) \cap \beta_{\scaleto{M,N \cap \beta}{5pt}}$ 
	is in $\sk(M)$. 
	If $N \cap \omega_1 = M \cap \omega_1$, then 
	$(N \cap \beta) \sim M$, so 
	$N \cap \beta_{\scaleto{M,N}{5pt}} = 
	(N \cap \beta) \cap \beta_{\scaleto{M,N \cap \beta}{5pt}} = 
	M \cap \beta_{\scaleto{M,N \cap \beta}{5pt}} = M \cap \beta_{\scaleto{M,N}{5pt}}$. 
	If $M \cap \omega_1 < N \cap \omega_1$, then 
	$M < (N \cap \beta)$, 
	so $M \cap \beta_{\scaleto{M,N}{5pt}} = 
	M \cap \beta_{\scaleto{M,N \cap \beta}{5pt}} \in 
	\sk(N \cap \beta) \subseteq \sk(N)$.
\end{proof}

\section{The Second Forcing}

In this section, we introduce the second main forcing poset $\p$ of the article. 
A condition in $\p$ will consist of a working part, which is a member of $\p^*$, 
together with an adequate set as a side condition. 
The next definition describes the required interaction between the working 
part and the side condition.

\begin{definition}[$A$-Separation]
	Let $T$ be a standard finite tree, 
	let $W$ be a subtree function on $T$, 
	and let $A$ be adequate. 
	We say that $W$ is \emph{$A$-separated} if 
	whenever $M \in A$, 
	$\eta$ and $\xi$ are distinct elements of $M \cap \dom(W)$, 
	and $x \in W(\xi) \cap W(\eta)$, then $x \in M$.
\end{definition}

\begin{definition}
	Let $\p$ be the forcing poset consisting 
	of all quadruples $(T,W,D,A)$ such that:
\begin{enumerate}
	\item $(T,W,D) \in \p^*$;
	\item $A$ is adequate;
	\item $W$ is $A$-separated.
\end{enumerate}
Let $(U,Y,E,B) \le (T,W,D,A)$ in $\p$ 
if $(U,Y,E) \le (T,W,D)$ in $\p^*$ and $A \subseteq B$.
\end{definition}

\begin{notation}
	For any $p \in \p$, we write $(T_p,W_p,D_p,A_p)$ for $p$.
\end{notation}

Note that $\p \subseteq H(\kappa)$.

\begin{lemma}
	For any $p \in \p$, there exists $q \le p$ such that 
	$T_q$ is downwards closed and has minimal splits.
\end{lemma}

\begin{proof}
	Fix $(T,W,D) \le (T_p,W_p,D_p)$ in $\p^*$ 
	satisfying the properties described in Lemma 2.11. 
	We claim that $W$ is $A_p$-separated, which easily implies 
	that $q = (T,W,D,A_p)$ is as required. 
	So let $M \in A_p$, let $\eta$ and $\xi$ be distinct elements of $M \cap \dom(W)$, 
	and let $x \in W(\eta) \cap W(\xi)$. 
	We claim that $x \in M$. 
	By Lemma 2.11,
	fix $z \in W_p(\eta) \cap W_p(\xi)$ such that $x \le_T z$. 
	Since $W_p$ is $A_p$-separated, $z \in M \cap \omega_1$. 
	As $x \le z$, $x \in M \cap \omega_1$ as well.
\end{proof}

\begin{definition}
	For any $p \in \p$ and for any $N \in \mathcal X$, 
	define $p + N = (T_p,W_p,D_p,A_p \cup \{ N \})$.
\end{definition}

\begin{lemma}
	For any $p \in \p$ and for any $N \in \mathcal X$ with $p \in \sk(N)$, 
	$p + N$ is in $\p$ and is an extension of $p$.
\end{lemma}

\begin{proof}
	The proof is easy using Lemmas 5.6.
\end{proof}

\begin{lemma}
	Suppose that $p \in \p$ and $N \in A_p$. 
	Define $C = A \cup \{ M \cap N : M \in A, \ M < N \}$. 
	Then $C$ is $N$-closed and $(T_p,W_p,D_p,C)$ is in $\p$ and extends $p$.
\end{lemma}

\begin{proof}
	By Lemma 5.10, $C$ is adequate and $N$-closed. 
	It suffices to prove that $W_p$ is $C$-separated. 
	Let $M \in C$, let $\eta$ and $\xi$ be distinct elements of $\dom(W_p) \cap M$, 
	and let $x \in W_p(\eta) \cap W_p(\xi)$. 
	If $M \in A_p$, then since $W_p$ is $A_p$-separated, $x \in M$. 
	Otherwise, for some $K \in A_p$ with $K < N$, $M = K \cap N$. 
	But then $\eta$ and $\xi$ are in $K$, so 
	$x \in K \cap \omega_1 = M \cap \omega_1$. 
\end{proof}

\begin{lemma}
	Suppose that $p \in \p$ and $\beta \in \Lambda$. 
	Define $C = A_p \cup \{ M \cap \beta : M \in A \}$. 
	Then $C$ is $\beta$-closed and 
	$(T_p,W_p,A_p,C)$ is in $\p$ and extends $p$. 
	Moreover, if $N \in A_p$ and $A_p$ is $N$-closed, then $C$ is $N$-closed.
\end{lemma}

\begin{proof}
	Similar to the proof of Lemma 6.7 using Lemma 5.14.
\end{proof}	

\begin{definition}
	Let $1 < d < \omega$. 
	Let $p_0,\ldots,p_{d-1}$ be in $\p$. 
	Define $p_0 \oplus \cdots \oplus p_{d-1}$ to be the quadruple $(T,W,D,A)$ satisfying:
	\begin{enumerate}
		\item $(T,W,D) = (T_{p_0},W_{p_0},D_{p_0}) \oplus \cdots \oplus 
		(T_{p_{d-1}},W_{p_{d-1}},D_{p_{d-1}})$;
		\item $A = A_{p_0} \cup \cdots \cup A_{p_{d-1}}$.
	\end{enumerate}
\end{definition}

The question of when $p_0 \oplus \cdots \oplus p_{d-1}$ is a condition extending 
each of $p_0,\dots,p_{d-1}$, in $\p$ or in quotients of $\p$, is one of 
the central issues we deal with in this article.

The next lemma is critical for analyzing quotients of $\p$ in later sections.

\begin{lemma}
	Let $1 < d < \omega$. 
	Let $p_0,\ldots,p_{d-1}$ be in $\p$. 
	Suppose that $p_0 \oplus \cdots \oplus p_{d-1}$ is a condition in $\p$ which extends 
	each of $p_0,\ldots,p_{d-1}$. 
	Assume that:
	\begin{enumerate}
		\item $r$ is in $\p$;
		\item $r \le p_0,\ldots,p_{d-1}$;
		\item for all $i < j < d$, for all $x \in T_{p_i} \setminus T_{p_j}$ 
		and for all $y \in T_{p_j} \setminus T_{p_i}$, 
		$x$ and $y$ are incomparable in $T_r$.
	\end{enumerate}
	Then $r \le p_0 \oplus \cdots \oplus p_{d-1}$.
\end{lemma}

\begin{proof}
	For each $i < d$, write $p_i = (T_i,W_i,D_i,A_i)$, and write 
	$p_0 \oplus \cdots \oplus p_{d-1} = (T,W,D,A)$. 
	By (2), $A_0,\ldots,A_{d-1}$ are subsets of $A_r$. 
	Hence, $A \subseteq A_r$. 
	We claim that $(T_r,W_r,D_r) \le (T,W,D)$ in $\p^*$. 
	By (2), $D \subseteq D_r$, $T \subseteq T_r$, and 
	$<_T \ \subseteq \ <_r$.
	
	To see that $T_r$ is an end-extension of $T$, 
	suppose that $x <_r y$ where $x, y \in T$, and we prove that $x <_T y$. 
	Fix $i, j < d$ such that $x \in T_i$ and $y \in T_j$. 
	If $i = j$, then we done since $r \le p_i$. 
	Assume that $i \ne j$. 
	By (3) and the fact that $x <_r y$, it cannot be the case that 
	both $x \in T_i \setminus T_j$ and $y \in T_j \setminus T_i$. 
	Without loss of generality, assume that $x \in T_i \cap T_j$. 
	Then $x$ and $y$ are in $T_j$, and since $r \le p_j$, it follows 
	that $x <_{T_j} y$ and hence $x <_T y$.

	By (2), $\dom(W) \subseteq \dom(W_r)$. 
	Consider $\eta \in \dom(W)$. 
	Then $W(\eta) = W_0(\eta) \cup \cdots \cup W_{d-1}(\eta)$. 
	Since $r \le p_0,\ldots,p_{d-1}$, for all $i < d$, $W_i(\eta)$ is a subset of $W(\eta)$. 
	So $W(\eta) \subseteq W_r(\eta)$. 
	Now assume that $\{ \eta, \xi \} \in D$ and 
	$x \in W_r(\eta) \cap W_r(\xi)$. 
	We will find some $z \in W(\eta) \cap W(\xi)$ such that $x \le_r z$.  
	Fix $i < d$ such that $\{ \eta, \xi \} \in D_i$. 
	Since $r \le p_i$, there exists some $z \in W_i(\eta) \cap W_i(\xi)$ 
	such that $x \le_r z$. 
	Then $z \in W(\eta) \cap W(\xi)$ and we are done.
\end{proof}

\section{Properness and Collapsing}

In this section, we prove that the forcing poset $\p$ is proper and 
collapses cardinals larger than $\omega_1$ and less than $\kappa$. 
Lemma 7.1 describes properties which are sufficient for amalgamating 
conditions over countable elementary substructures. 
The case that $d = 2$ is used in Theorem 7.2 
to prove that $\p$ is proper, and the general 
case that $d \ge 2$ is used in Section 7 to prove that $\p$ is Y-proper. 
Lemma 7.1 is also used implicitly in proving that quotients of $\p$ are Y-proper, 
by way of Lemma 8.6.

\begin{lemma}
	Let $1 < d < \omega$. 
	Let $p_0,\ldots,p_{d-1}$ be in $\p$. 
	Write $p_i = (T_i,W_i,D_i,A_i)$ for all $i < d$. 
	Assume that for all $i < d$, $N_i \in A_i$, $A_i$ is $N_i$-closed, 
	and for all $i < j < d$, $p_i \in \sk(N_j)$. 
	Let $\delta_i = N_i \cap \omega_1$ for all $i < d$.
	
	Assume that there exist commutative families of functions 
	$\{ f_{j,i} : i < j < d \}$ and 
	$\{ g_{j,i} : i < j < d \}$ such that for each $i < j < d$, 
	$f_{j,i} : \dom(W_j) \to \dom(W_{i})$ and $g_{j,i} : A_j \to A_{i}$ are bijective. 
	Finally, assume that the following statements hold for all $i < j < d$:
	\begin{enumerate}
		\item $T_{i} \res \delta_i = T_{j} \res \delta_j$;
		\item $\dom(W_i) \cap N_i = \dom(W_j) \cap N_j$ and for all $\eta$ 
		in this set, $f_{j,i}(\eta) = \eta$;
		\item for all $\eta \in \dom(W_j)$ and for all $M \in A_j$, 
		$\eta \in M$ iff $f_{j,i}(\eta) \in g_{j,i}(M)$;
		\item for all $\eta \in \dom(W_j)$, 
		$W_{i}(f_{j,i}(\eta)) \cap \delta_{i} = W_j(\eta) \cap \delta_j$;
		\item $A_i \cap \sk(N_i) = A_j \cap \sk(N_j)$ and for all $M$ 
		in this set, $g_{j,i}(M) = M$;
		\item for all $M \in A_j$, if $M \cap \omega_1 < \delta_j$, then 
		$g_{j,i}(M) \cap \omega_1 = M \cap \omega_1$ 
		and $M \cap N_j \subseteq g_{j,i}(M)$.
	\end{enumerate}
	Then $p_0 \oplus \cdots \oplus p_{n-1}$ is a condition which extends 
	$p_0,\ldots,p_{n-1}$.
\end{lemma}

\begin{proof}
	Observe that by elementarity, $\delta_i \in C_h$ for all $i < d$. 
	Note that for all $i < j < d$, by (1) and the fact that 
	$p_i \in \sk(N_j)$, $T_i \subseteq \delta_j$ and 
	$T_i \cap T_j \subseteq \delta_i$.

	\textbf{Claim 1:} For all $k < d$, 
	$\{ \dom(W_i) : i < d \}$ is a $\Delta$-system with root $\dom(W_k) \cap N_k$. 
	So for all $i < j < d$, 
	$\dom(W_i) \cap \dom(W_j) \subseteq N_i \cap N_j$.
	
	\emph{Proof:} This follows easily from (2) together with the fact that 
	for all $i < j < d$, $p_i \in \sk(N_j)$.
	
	\textbf{Claim 2:} For all $i < j < d$, 
	$(T_i,W_i,D_i)$ and $(T_j,W_j,D_j)$ are $(\delta_i,\delta_j)$-split.

	\emph{Proof:} We have that $T_i \res \delta_i = T_j \res \delta_j$ by (1), 
	and since $p_i \in \sk(N_j)$, $T_i \subseteq \delta_j$. 
	If $\eta \in \dom(W_i) \cap \dom(W_j)$, then by (2), $f_{j,i}(\eta) = \eta$. 
	Hence by (4), $W_i(\eta) \cap \delta_i = W_j(\eta) \cap \delta_j$. 
	Finally, consider distinct $\eta, \xi \in \dom(W_i) \cap \dom(W_j)$. 
	Then $\eta$ and $\xi$ are in $N_i \cap N_j$. 
	Since $W_i$ is $A_i$-separated and $N_i \in A_i$, it follows that 
	$W_i(\eta) \cap W_i(\xi) \subseteq N_i \cap \omega_1 = \delta_i$. 
	And because $W_j$ is $A_j$-separated and $N_j \in A_j$, 
	$W_j(\eta) \cap W_j(\xi) \subseteq N_j \cap \omega_1 = \delta_j$. 
	This completes the proof of the claim.
	
	Define $r = p_0 \oplus \cdots \oplus p_{d-1}$, 
	which we denote by $(T,W,D,A)$. 
	We prove that $r \in \p$ and $r$ is an extension of $p_i$ for all $i < d$. 
	By claims 1 and 2 and Lemma 2.15, 
	$(T,W,D)$ is in $\p^*$ and is an 
	extension of $(T_i,W_i,D_i)$ for all $i < d$. 
	Obviously, $A_i \subseteq A$ for all $i < d$. 
	Using (5), it is easy to check that the assumptions of Corollary 5.12 hold for 
	$A_0,\ldots,A_{d-1}$ and $N_1,\ldots,N_{d-1}$, 
	so $A = A_0 \cup \cdots \cup A_{d-1}$ is adequate.

	Finally, we prove that $W$ is $A$-separated. 
	Let $M \in A$, 
	let $\eta$ and $\xi$ be distinct elements of $M \cap \dom(W)$, 
	and let $x \in W_r(\eta) \cap W_r(\xi)$. 
	We show that $x \in M$. 
	If there exists some $i < d$ such that 
	$M \in A_i$ and $x \in W_i(\eta) \cap W_i(\xi)$, then we are done since $p_i$ 
	is a condition. 
	So assume not. 
	Fix $i, j, k < d$ such that $M \in A_i$, $x \in W_j(\eta)$, and $x \in W_k(\xi)$.

	\emph{Case 1:} $j = k$, $j < i$, and $M \in A_i \setminus \sk(N_i)$. 			
	Then $p_j \in \sk(N_i)$, so $x \in T_j \subseteq \delta_i$ and 
	$\eta$ and $\xi$ are in $N_i$. 
	And $x \in W_j(\eta) \cap W_j(\xi)$. 
	If $N_i \le M$, then $x \in \delta_i \subseteq M$ 
	and we are done. 
	Otherwise, $M \cap \omega_1 < \delta_i$. 
	By (6), $M \cap \omega_1 = g_{i,j}(M) \cap \omega_1$ and 
	$\eta$ and $\xi$ are in $M \cap N_i \subseteq g_{i,j}(M)$. 
	Since $W_j$ is $A_j$-separated and $g_{i,j}(M) \in A_j$, 
	$x \in g_{i,j}(M) \cap \omega_1 \subseteq M$.

	\emph{Case 2:} $j = k$, $j < i$, and $M \in A_i \cap \sk(N_i)$. 
	By (5), $M \in A_j$. 
	So we are done since $W_j$ is $A_j$-separated.
	
	\emph{Case 3:} $j = k$ and $i < j$. 
	Then $p_i \in \sk(N_j)$, so $M \in \sk(N_j)$. 
	As $\eta, \xi \in M$, it follows that $\eta, \xi \in N_j$. 
	So by (2), $\eta, \xi \in \dom(W_j) \cap N_j \subseteq \dom(W_i)$. 
	By Definition 2.13(4), $W_j(\eta) \cap W_j(\xi) \subseteq \delta_j$, and 
	hence $x < \delta_j$. 
	By Definition 2.13(3), $W_i(\eta) \cap \delta_i = W_j(\eta) \cap \delta_j$ 
	and $W_i(\xi) \cap \delta_i = W_j(\xi) \cap \delta_j$. 
	So $x \in W_i(\eta) \cap W_i(\xi)$. 
	Since $W_i$ is $A_i$-separated and $M \in A_i$, $x \in M$.
	
	In the remaining cases, $j \ne k$. 
	Without loss of generality assume that $j < k$.
	
	\emph{Case 4:} $j \ne k$ and either 
	$\eta \in \dom(W_k)$ or $\xi \in \dom(W_j)$. 
	Assume that $\eta \in \dom(W_k)$. 
	Then $\xi \in \dom(W_k)$ and $\eta \in \dom(W_j) \cap \dom(W_k)$. 
	By Lemma 2.14(c), $W_k(\xi) \cap W_j(\eta) \subseteq W_k(\eta)$. 
	Hence, $x \in W_k(\eta) \cap W_k(\xi)$. 
	So we are back to the situation of Cases 1, 2, and 3, and we are done. 
	The case that $\xi \in \dom(W_j)$ is similar using Lemma 2.14(b).
	
	For the remaining cases, we may assume that 
	$\eta \in \dom(W_j) \setminus \dom(W_k)$ and 
	$\xi \in \dom(W_k) \setminus \dom(W_j)$. 
	Note that 
	$x \in W_j(\eta) \cap W_k(\xi) \subseteq 
	T_j \cap T_k \subseteq \delta_j$. 
	So $x < \delta_j$. 	
	By (4), $W_j(f_{k,j}(\xi)) \cap \delta_j = W_k(\xi) \cap \delta_k$. 
	So $x \in W_j(f_{k,j}(\xi))$. 
	If $\delta_j \le M \cap \omega_1$, then we are done. 
	So assume that $M \cap \omega_1 < \delta_j$. 

	\emph{Case 5:} $k < i$. 
	Then $M \cap \omega_1 < \delta_i$. 
	And $p_k \in \sk(N_i)$ implies that 
	$\xi \in M \cap N_i \subseteq g_{i,k}(M)$ by (6). 
	So by (3) and commutativity, 
	$f_{k,j}(\xi) \in g_{k,j}(g_{i,k}(M)) = g_{i,j}(M)$. 
	And $p_j \in \sk(N_i)$ and $M \cap \omega_1 < \delta_i$ implies by (6) that 
	$\eta \in M \cap N_i \subseteq g_{i,j}(M)$. 
	So $x \in W_j(\eta) \cap W_j(f_{k,j}(\xi))$, 
	$\eta$ and $f_{k,j}(\xi)$ are in $g_{i,j}(M)$, and $g_{i,j}(M) \in A_j$. 
	As $W_j$ is $A_j$-separated, it follows that 
	$x \in g_{i,j}(M) \cap \omega_1 \subseteq M$.

	\emph{Case 6:} $k = i$. 
	By (3), $f_{k,j}(\xi) \in g_{k,j}(M) = g_{i,j}(M)$. 
	Also, $p_j \in \sk(N_i)$. 
	By (6), $\eta \in \dom(W_j) \cap M \subseteq N_i \cap M \subseteq g_{i,j}(M)$. 
	So $x \in W_j(\eta) \cap W_j(f_{k,j}(\xi))$, 
	and $\eta$ and $f_{k,j}(\xi)$ are in $g_{i,j}(M)$. 
	Since $W_j$ is $A_j$-separated and $g_{i,j}(M) \in A_j$, 
	$x \in g_{i,j}(M) \cap \omega_1 \subseteq M$.

	\emph{Case 7:} $i < k$. 
	Then $M \in A_i \subseteq \sk(N_k)$. 
	So $\xi \in M \subseteq N_k$. 
	By (2), $\xi \in \dom(W_k) \cap N_k = \dom(W_j) \cap N_j$, which 
	contradicts our assumption that $\xi \notin \dom(W_j)$.
\end{proof}

\begin{thm}
	Let $\chi > \kappa$ be regular. 
	Let $M$ be a countable elementary 
	substructure of $\mathcal B = (H(\chi),\in,\psi,\p)$ such that 
	$N = M \cap \kappa \in \mathcal X$. 
	Then for any $u \in M \cap \p$, $u + N$ is in $\p$, $u + N \le u$, 
	and $u + N$ is $(M,\p)$-generic. 
	In fact, for any dense open set $\mathcal{D} \subseteq \p$ in $M$, 
	for any $q \le u + N$ in $\mathcal D$ such that $A_q$ is $N$-closed, 
	there exists some $\bar q \in M \cap \mathcal{D}$ such that 
	$\bar q \oplus q$ is in $\p$ and extends $\bar q$ and $q$. 
\end{thm}

\begin{proof}
	By Lemma 5.7, $M \cap H(\kappa) = \sk(N)$. 
	So $u \in \sk(N)$. 
	By Lemma 6.6, $u + N$ is in $\p$ and extends $u$. 
	Fix a dense open set $\mathcal{D} \subseteq \p$ in $M$. 
	Fix $q \le u + N$ in $\mathcal{D}$ which is $N$-closed.

	Let $N_q = N$ and let $\delta_q = N \cap \omega_1$. 
	Enumerate $\dom(W_q) = \{ \eta_0,\ldots,\eta_{m-1} \}$ and 
	$A_q = \{ M_0,\ldots,M_{n-1} \}$, where $m, n < \omega$ and $M_0 = N$. 
	Define:
	\begin{itemize}
		\item $U_0 = \{ k < m : \eta_k \in N \}$;
		\item $U_1 = \{ l < n : M_l \in \sk(N) \}$;
		\item $U_2 = \{ l < n : M_l < N \}$;
		\item $U_3 = \{ (k,l) \in m \times n : \eta_k \in M_l \}$.
	\end{itemize}
	
	Define a formula with free variables 
	$$
	\varphi = \varphi(\dot q,\dot T,\dot W,\dot D,\dot A,\dot \delta, 
	\dot \eta_0,\ldots,\dot \eta_{m-1},\dot M_0,\ldots,\dot M_{n-1})
	$$
	to be conjunction of the following:
	\begin{enumerate}
		\item $\dot q = (\dot{T},\dot{W},\dot{D},\dot A) \in \mathcal{D}$;
		\item $\dom(\dot{W}) = \{ \dot \eta_0,\ldots,\dot \eta_{m-1} \}$ 
		and $\dot A = \{ \dot M_0,\ldots,\dot M_{n-1} \}$;
		\item $\dot \delta = \dot M_0 \cap \omega_1$;
		\item $\dot{T} \res \dot \delta = T_q \res \delta_q$;
		\item for all $k \in U_0$, $\dot \eta_k = \eta_k$;
		\item for all $(k,l) \in m \times n$, 
		$(k,l) \in U_3$ iff $\dot \eta_k \in \dot M_l$;
		\item $\dom(\dot{W}) \cap \dot M_0 = \dom(W_q) \cap N$;
		\item for all $k < m$, 
		$\dot{W}(\dot \eta_k) \cap \dot \delta = 
		W_{q}(\eta_k) \cap \delta_q$;
		\item $\dot A \cap \sk(\dot M_0) = A_q \cap \sk(N)$;
		\item for all $l \in U_1$, $\dot M_l = M_l$;
		\item for all $l \in U_2$, 
		$\dot M_l \cap \omega_1 = M_l \cap \omega_1$ 
		and $M_l \cap N \subseteq \dot M_l$;
		\item for all $l < n$, $l \in U_2$ iff $\dot M_l < \dot M_0$;
		\item for all $l \in U_2$, $\dot M_l \cap \dot M_0 \in \dot A$.
	\end{enumerate}
	
	It is routine to check that all of the parameters appearing in $\varphi$ 
	are members of $M$, and that 
	$$
	\mathcal B \models \varphi[q, T_q, W_q, D_q, A_q, N \cap \omega_1, 
	\eta_0,\ldots,\eta_{m-1},M_0,\ldots,M_{n-1}].
	$$
	By elementarity, we can find objects 
	$\bar q$, $\bar T$, $\bar W$, $\bar D$, $\bar A$, $\bar \delta$, 
	$\bar \eta_0,\ldots,\bar \eta_{m-1}$, and $\bar M_0,\ldots,\bar M_{n-1}$ 
	in $M$ which satisfy the same.

	Now it is straightforward to check that the assumptions of Lemma 7.1 
	are satisfied, where $d = 2$, $\bar q$ and $\bar M_0$ 
	serve the roles of $p_0$ and $N_0$, $q$ and $N$ serve the roles of 
	$p_1$ and $N_1$, and the functions $f_{1,0}$ and $g_{1,0}$ 
	are defined by $f_{1,0}(\eta_k) = \bar \eta_k$ for all $k < m$ 
	and $g_{1,0}(M_l) = \bar M_l$ for all $l < n$. 
	It follows that $\bar q \oplus q$ is a condition which extends 
	$\bar q$ and $q$. 
\end{proof}

Since $\mathcal X$ is a club subset of $[\kappa]^\omega$, we immediately 
have the following corollary.

\begin{corollary}
	The forcing poset $\p$ is proper.
\end{corollary}

Since $\p$ is proper, it preserves $\omega_1$. 
Let us check that $\p$ collapses every cardinal $\mu$ such that $\omega_1 < \mu < \kappa$.

\begin{proposition}
	Suppose that $\mu$ is a cardinal and $\omega_1 < \mu < \kappa$. 
	Then $\p$ forces that $\mu$ is not a cardinal.
\end{proposition}

\begin{proof}
	Let $G$ be a generic filter on $\p$. 
	In $V[G]$, define 
	$$
	Z = \{ M \in \mathcal X : \exists p \in G \ (M \in A_p \ \land \ \mu \in M) \}.
	$$
	If $M$ and $N$ are in $Z$, then $\mu \in M \cap N \cap \kappa$, and hence 
	$\beta_{\scaleto{M,N}{5pt}} > \mu$. 
	It easily follows from the definition of $\p$ that 
	for all $M, N \in Z$, $M \cap \omega_1 < N \cap \omega_1$ iff 
	$M \cap \mu \subsetneq N \cap \mu$. 
	So $Z_\mu = \{ M \cap \mu : M \in Z \}$ is a well-ordered chain of 
	countable sets with order-type at most $\omega_1$. 
	It easily follows from Lemma 6.6 that the union of this chain is equal to $\mu$, 
	and hence this chain has order type equal to $\omega_1$ since $\omega_1$ is preserved. 
	So $\mu$ is the union of $\omega_1$-many countable sets 
	and therefore has size $\omega_1$.	
\end{proof}

We now briefly discuss the generic object which is added by $\p$.

\begin{definition}
	Let $G$ be a generic filter on $\p$. 
	Define $(T_G,<_G)$ by:
	\begin{itemize}
		\item $x \in T_G$ if there exists some $p \in G$ such that $x \in T_p$;
		\item $x <_G y$ if there exists some $p \in G$ such that $x <_p y$.
	\end{itemize}
	For any $\eta < \kappa$, define 
	$W_G(\eta) = \bigcup \{ W_p(\eta) : p \in G, \ \eta \in \dom(W_p) \}$.
\end{definition}

As usual, we abbreviate $(T_G,<_G)$ by $T_G$. 
We occasionally write $\dot G$ for the canonical $\p$-name for a 
generic filter on $\p$.
Let $T_{\dot G}$ be a $\p$-name for the above object.

The following proposition has almost the same proof as the analogous fact 
about $\p'$ from Section 3.

\begin{proposition}
	Let $G$ be a generic filter on $\p$. 
	Then $T_G$ is a normal infinitely splitting $\omega_1$-tree and 
	$\{ W_G(\eta) : \eta < \kappa \}$ is a pairwise strongly almost disjoint 
	family of uncountable downwards closed subtrees of $T_G$.
\end{proposition}

\begin{proposition}
	The forcing poset $\p$ forces that $T_{\dot G}$ is Aronszajn.
\end{proposition}

\begin{proof}
	Suppose for a contradiction that $u \in \p$ forces that 
	$\dot b$ is a cofinal branch of $T_{\dot G}$. 
	Fix a regular cardinal $\chi > \kappa$ such that $\dot b \in H(\chi)$. 
	Let $M$ be a countable elementary substructure of $(H(\chi),\in,\psi,\p)$ 
	such that $u$ and $\dot b$ are in $M$ and $N = M \cap \kappa \in \mathcal X$. 
	By Theorem 7.2, $u + N$ is a condition extending $u$ which is $(M,\p)$-generic.

	Fix $q \le u + N$ and $x_q \ge N \cap \omega_1$ 
	such that $x_q \in T_q$ and $q$ forces that $x_q \in \dot b$. 
	By extending $q$ further if necessary, we may assume that $A_q$ is $N$-closed. 
	Fix $\delta < N \cap \omega_1$ such that $T_q \cap (N \cap \omega_1) \subseteq \delta$. 
	Let $\mathcal{D}$ be the set of conditions $s \in \p$ such that for some 
	$x_s \ge \delta$ in $T_s$, $s$ forces that $x_s \in \dot b$. 
	Note that $\mathcal{D}$ is dense open, $\mathcal{D} \in M$, and $q \in \mathcal{D}$.
	
	By Theorem 7.2, fix $\bar q$ in $\mathcal{D} \cap M$ such that $r = \bar q \oplus q$ 
	is in $\p$ and is an extension of $\bar q$ and $q$. 
	As $\bar q \in \mathcal{D} \cap M$, fix $x_{\bar q} \ge \delta$ in 
	$T_{\bar q} \cap M$ such that $\bar q$ forces that $x_{\bar q} \in \dot b$. 
	Now $T_q \cap (N \cap \omega_1) \subseteq \delta$ and $\bar q \in M$ imply that 
	$x_{\bar q} \in T_{\bar q} \setminus T_q$ and $x_q \in T_q \setminus T_{\bar q}$. 
	By the definition of $T_{\bar q} \oplus T_q$, $x_{\bar q}$ and $x_q$ 
	are not comparable in $T_r$. 
	For all $s \le r$, $T_s$ is an end-extension of $T_r$ and hence  
	$x$ and $y$ are incomparable in $T_s$. 
	So $r$ forces that $x$ and $y$ are incomparable in $T_{\dot G}$, which 
	contradicts that $r$ forces that $x$ and $y$ are both in $\dot b$.
\end{proof}

\section{Y-Properness}

In our applications of the forcing poset $\p$, we need to know that quotients of 
$\p$ have the $\omega_1$-approximation property, and in particular, that they   
do not add new cofinal branches of trees with height $\omega_1$. 
The key to this fact is the property of Y-properness due to 
Chodounsk\'{y} and Zapletal (\cite{CZ}).

\begin{definition}
	A forcing poset $\q$ is \emph{Y-proper} if for all large enough 
	regular cardinals $\chi$ with $\q \in H(\chi)$, there are club 
	many $M \in [H(\chi)]^\omega$ such that $M$ is an elementary 
	substructure of $(H(\chi),\in,\q)$, and for all $p \in M \cap \q$ 
	there exists $q \le p$ which is $(M,\q)$-generic and satisfies that 
	for all $r \le q$, there exists a filter $\mathcal F \in M$ on the Boolean 
	completion $\mathcal B(\q)$ such that 
	$\{ s \in M \cap \mathcal B(\q) : r \le s \} \subseteq \mathcal F$. 
	If the above holds for stationarily many (rather than club many) 
	$M$ in $[H(\chi)]^\omega$, then we say that $\q$ is 
	\emph{Y-proper on a stationary set}.
\end{definition}

Note that Y-proper implies proper. 
Recall that a forcing poset $\q$ has the \emph{$\omega_1$-approximation property} 
if whenever $X \in V$, $B \subseteq X$ is in $V^\q$, and for all countable 
$a \subseteq X$ in $V$, $B \cap a \in V$, then $B \in V$. 
If $\q$ has the $\omega_1$-approximation property, then for any regular 
uncountable cardinal $\mu$, $\q$ does not add new cofinal branches to 
any tree with height $\mu$. 
Our interest in Y-properness comes from the following consequence of it 
(see \cite[Corollary 4.1]{CZ} and the proof of \cite[Theorem 2.8]{CZ}).

\begin{thm}
	Suppose that $\q$ is a forcing poset which is Y-proper on a stationary set. 
	Then $\q$ has the $\omega_1$-approximation property.
\end{thm}

In Section 12, we also make use of the Y-c.c.\ property of a forcing poset (\cite{CZ}). 
The only things which the reader needs to know about this property is that 
it implies Y-properness and it is preserved under finite support forcing iterations.

We now proceed towards proving that $\p$ is Y-proper.

\begin{notation}
	For any $p \in \p$ define:
	\begin{itemize}
		\item $m_p = |\dom(W_p)|$ and $n_p = |A_p|$;
		\item $\langle \eta^p_i : i < m_p \rangle$ is the unique 
		enumeration of $\dom(W_p)$ such that 
		$\eta^p_i < \eta^p_j$ for all $i < j < m_p$;
		\item $\langle K^p_i : i < n_p \rangle$ is the unique enumeration of $A_p$ 
		such that $K^p_i \lhd K^p_j$ for all $i < j < n_p$.
	\end{itemize} 
\end{notation}

\begin{definition}
	Define a function $w$ as follows. 
	The domain of $w$ is the set of ordered pairs $(p, N)$ such that 
	$p \in \p$, $N \in A_p$, and $A_p$ is $N$-closed. 
	For any such $(p,N)$, letting $\delta = N \cap \omega_1$, define 
	$$
	w(p, N) = (t,a,b,m,n,w_0,\ldots,w_{m-1},U_0,U_1,U_2,U_3,h_0,h_1),
	$$
	where:
	\begin{enumerate}
		\item[(a)] $t = T \res \delta$;
		\item[(b)] $a = \dom(W) \cap N$;
		\item[(c)] $b = A \cap \sk(N)$;
		\item[(d)] $m = m_p$ and $n = n_p$;
		\item[(e)] $w_k = W(\eta^p_k) \cap \delta$ for all $k < m$;
		\item[(f)] $U_0 = \{ k < m : \eta^p_k \in N \}$;
		\item[(g)] $U_1 = \{ l < n : K^p_l \in \sk(N) \}$;
		\item[(h)] $U_2 = \{ l < n : K^p_l \cap \omega_1 < \delta \}$.
		\item[(i)] $U_3 = \{ (k,l) \in m \times n : \eta^p_k \in K^p_l \}$;
		\item[(j)] $h_0 : U_2 \to \delta$ is a function and 
		for all $l \in U_2$, $h_0(l) = K^p_l \cap \omega_1$;
		\item[(k)] $h_1 : U_2 \to \sk(N)$ is a function and 
		for all $l \in U_2$, $h_1(l) = K^p_l \cap N$.
	\end{enumerate}
\end{definition}

\begin{lemma}
	For all $(p, N)$ in the domain of $w$, 
	$w(p, N) \in \sk(N)$.
\end{lemma}

\begin{proof}
	Every member of the tuple $w(p,N)$ is a finite subset of $\sk(N)$.
\end{proof}

\begin{lemma}
	Let $1 < d < \omega$. 
	Let $p_0,\ldots,p_{d-1}$ be in $\p$, let $N_0,\ldots,N_{d-1}$ be in $\mathcal X$, 
	and suppose that for all $i < d$, $(p_i,N_i) \in \dom(w)$. 
	Assume that for all $i < j < d$, $w(p_i, N_i) = w(p_j, N_j)$ and $p_i \in \sk(N_j)$. 
	Then $p_0 \oplus \cdots \oplus p_{d-1}$ is a condition which extends 
	each of $p_0,\ldots,p_{d-1}$.
\end{lemma}

\begin{proof}
	Let $m = m_{p_i}$ and $n = n_{p_i}$ for some (any) $i < d$. 
	For each $i < j < d$, define $f_{j,i} : \dom(W_j) \to \dom(W_i)$ by 
	letting $f_{j,i}(\eta_k^{p_j}) = \eta_k^{p_i}$ for all $k < m$, 
	and define $g_{j,i} : A_j \to A_i$ by 
	letting $g_{j,i}(K_l^{p_j}) = K_l^{p_i}$ for all $l < n$. 
	Clearly, this definition gives commutative families of bijections.
	
	We verify properties (1)-(6) of Lemma 7.1  
	using (a)-(k) of Definition 8.4, with the other required properties 
	being immediate. 
	(1) follows from (a). 
	(2) follows from (b) and (f). 
	(3) follows from (i). 
	(4) follows from (e). 
	(5) follows from (c) and (g). 
	The first part of (6) follows from (h) and (j). 
	For the second part of (6), if $i < j < d$ and $l \in U_2$, 
	then by (k), 
	$K^{p_j}_l \cap N_j = h_1(l) = K^{p_i}_l \cap N_i \subseteq g_{j,i}(K^{p_j}_l)$.
	
	By Lemma 7.1, $p_0 \oplus \cdots \oplus p_{d-1}$ is in $\p$ 
	and extends each of $p_0,\ldots,p_{d-1}$.
\end{proof}

\begin{lemma}
	Under the assumptions of Lemma 8.6, for all $k < d$:
	\begin{itemize}
		\item $\{ T_{p_i} : i < d \}$ is a $\Delta$-system with 
		root $T_{p_k} \cap (N_k \cap \omega_1)$;
		\item $\{ \dom(W_{p_i}) : i < d \}$ is a $\Delta$-system 
		with root $\dom(W_{p_k}) \cap N_k$;
		\item $\{ A_{p_i} : i < d \}$ is a $\Delta$-system with root $A_{p_k} \cap \sk(N_k)$.
	\end{itemize}
\end{lemma}

\begin{proof}
	Straightforward using (a), (b), and (c) of Definition 8.4 
	together with the fact that $p_i \in \sk(N_j)$ whenever $i < j < d$.
\end{proof}

\begin{definition}
	Let $\vec z$ be in the range of $w$. 
	A set $R \subseteq \p$ is said to be \emph{$\vec z$-robust} 
	if the set 
	$$
	\{ N \in \mathcal X : \exists p \in R \ (w(p, N) = \vec z) \}
	$$
	is stationary in $[\kappa]^\omega$.
\end{definition}

\begin{proposition}
	For any $\vec z$ in the range of $w$, the collection 
	$\{ \ \sum R : R \subseteq \p \ \text{is $\vec z$-robust} \ \}$
	is centered.
\end{proposition}

\begin{proof}
	Let $1 < d < \omega$ and let 
	$R_0,\ldots,R_{d-1}$ be $\vec z$-robust sets. 
	We prove that there exists some $r \in \p$ such that for all $i < d$, $r \le \sum R_i$. 
	By induction, we choose $p_0,\ldots,p_{d-1}$ and $N_0,\ldots,N_{d-1}$ as follows. 
	Fix any $p_0 \in R_0$ and $N_0$ such that $w(p_0,N_0) = \vec z$. 
	Now let $0 < i < d$ and assume that $p_j$ and $N_j$ are defined for all $j < i$. 
	Since $R_i$ is $\vec z$-robust, by stationarity we can find some 
	$p_i \in R$ and $N_i$ such that 
	$w(p_i, N_i) = \vec z$ and 
	for all $j < i$, $p_j \in \sk(N_i)$. 
	This completes the induction. 
	By Lemma 8.6, $q = p_0 \oplus \cdots \oplus p_{d-1}$ is in $\p$ and extends each of 
	$p_0,\ldots,p_{d-1}$. 
	Hence, for all $i < d$, $q \le p_i \le \sum R_i$.
\end{proof}

\begin{thm}
	The forcing poset $\p$ is Y-proper.
\end{thm}

\begin{proof}
	Fix a regular cardinal $\chi > \kappa$. 
	Let $M$ be a countable elementary substructure of $(H(\chi),\in,\psi,\p)$ 
	such that $N = M \cap \kappa \in \mathcal X$. 
	Note that there are club many such $M$ in $[H(\chi)]^\omega$. 
	Consider $u \in M \cap \p$. 
	By Theorem 7.2, $u + N$ is a condition in $\p$ extending $u$ which is $(M,\p)$-generic. 
	Consider any condition $q \le u + N$. 
	We will find a filter $\mathcal F$ on $\mathcal B(\p)$ in $M$ such that 
	for every $s \in M \cap \mathcal B(\p)$, if $q \le s$ then $s \in \mathcal F$.
	
	Using Lemma 6.7, extend $q$ to $r$ such that $A_r$ is $N$-closed. 
	Then $(r, N)$ is in the domain of $w$. 
	Let $\vec z = w(r, N)$. 
	Then $\vec z \in \sk(N) \subseteq M$. 
	Define 
	$\mathcal F_0 = 
	\left\{ \sum R : R \subseteq \p \ \text{is $\vec z$-robust} \right\}$. 
	By Proposition 8.9, $\mathcal F_0$ is centered, and by elementarity, 
	$\mathcal F_0 \in M$. 
	Define $\mathcal F = \{ b \in \mathcal B(\p) : \exists c \in \mathcal F_0 \ (c \le b) \}$. 
	Then $\mathcal F$ is a filter on $\mathcal B(\p)$ and $\mathcal F \in M$.

	Suppose that $q \le s$ and $s \in M \cap \mathcal B(\p)$. 
	Define $R = \{ t \in \p : t \le s \}$. 
	Clearly, $s = \sum R$, $R \in M$, and $r \in R$. 
	We claim that $R$ is $\vec z$-robust, and therefore 
	$s = \sum R \in \mathcal F_0 \subseteq \mathcal F$. 
	Let $C$ be a club subset of $[\kappa]^\omega$ in $M$. 
	Then $N \in C$. 
	So $N \in C$, $r \in R$, and $w(r, N) = \vec z$. 
	By elementarity, it follows that the set of all $K \in [\kappa]^\omega$ 
	for which there exists some $t \in R$ such that $w(t, K) = \vec z$ 
	is stationary.
\end{proof}

\section{A Dense Set For Projecting}

The main goal for the remainder of the article is to prove that certain quotients 
of the forcing $\p$ are Y-proper in an intermediate extension. 
The proof of this fact is complex and will be completed in several steps 
over the next few sections. 
In this section, we identify a dense subset of $\p$ which we use in the 
next section to define a natural projection mapping.

\begin{definition}
	Define $\Sigma$ to be the set of $\theta \in \Lambda$ such that 
	$\sk(\theta)$ is an elementary substructure of 
	$(H(\kappa),\in,\psi,\mathcal X,\p)$.
\end{definition}

Recall that $\sk(\theta)$ denotes the closure of $\theta$ 
under the definable Skolem functions for the structure $\mathcal A = (H(\kappa),\in,\psi)$. 
If $\theta \in \Lambda$, then $\theta = \sk(\theta) \cap \kappa = \psi[\theta]$. 
The set $\Sigma$ is equal to a club subset of $\kappa$ intersected with 
$\kappa \cap \cof(> \! \omega)$. 

\begin{definition}
	Let $\theta \in \Sigma$. 
	Define $D_\theta$ to be the set of conditions $r \in \p$ satisfying 
	that $A_r$ is $\theta$-closed and there exist functions 
	$f : \dom(W_r) \setminus \theta \to \dom(W_r) \cap \theta$ and 
	$g : A_r \setminus \sk(\theta) \to A_r \cap \sk(\theta)$ satisfying:
	\begin{enumerate}
		\item[(a)] for all $\eta \in \dom(W_r) \setminus \theta$, 
		$W(\eta) = W(f(\eta))$;
		\item[(b)] for all $M \in A_r \setminus \sk(\theta)$, 
		$M \cap \omega_1 = g(M) \cap \omega_1$ 
		and $M \cap \theta \subseteq g(M)$;
		\item[(c)] for all $\eta \in \dom(W_r) \setminus \theta$ 
		and for all $M \in A_r \setminus \sk(\theta)$, 
		$\eta \in M$ iff $f(\eta) \in g(M)$;
		\item[(d)] for all $\eta \in \dom(W_r) \setminus \theta$ 
		and for all $\xi \in \dom(W_r) \cap \theta$, 
		if $\{ \eta, \xi \} \in D_r$ then $\{ f(\eta), \xi \} \in D_r$;
		\item[(e)] for all $\eta, \xi \in \dom(W_r) \setminus \theta$, 
		if $\{ \eta, \xi \} \in D_r$ then $\{ f(\eta), f(\xi) \} \in D_r$;
		\item[(f)] if $K, M \in A_r \setminus \sk(\theta)$ and $K \subseteq M$, 
		then $g(K) \subseteq g(M)$.  
	\end{enumerate}
\end{definition}

\begin{lemma}
	Let $\theta \in \Sigma$. 
	Suppose that $q = (T,W,D,A) \in \p$ and $A$ is $\theta$-closed. 
	Assume that $\bar q = (\bar T,\bar W,\bar D,\bar A) \in \sk(\theta) \cap \p$, 
	where $\bar T = T$, 
	and there exist bijections 
	$f : \dom(W) \to \dom(\bar W)$ and 
	$g : A \to \bar A$ satisfying:
	\begin{enumerate}
		\item for all $\eta \in \dom(W) \cap \theta$, $f(\eta) = \eta$;
		\item for all $M \in A \cap \sk(\theta)$, $g(M) = M$;
		\item for all $\eta \in \dom(W)$, $W(\eta) = \bar{W}(f(\eta))$;
		\item for all $M \in A$, $M \cap \omega_1 = g(M) \cap \omega_1$ 
		and $M \cap \theta \subseteq g(M)$;
		\item for all $\eta \in \dom(W)$ and for all 
		$M \in A$, $\eta \in M$ iff $f(\eta) \in g(M)$.
	\end{enumerate}
	Then $\bar q \oplus q$ is a condition in $\p$ which extends $\bar q$ and $q$.
\end{lemma}

\begin{proof}
	Write $\bar q \oplus q = (T,Y,E,C)$. 
	Then $\dom(Y) = \dom(\bar W) \cup \dom(W)$. 
	By (1) and (2), $\dom(W) \cap \theta \subseteq \dom(\bar W)$ and 
	$A \cap \sk(\theta) \subseteq \bar A$. 
	By (1) and (3), for all $\eta \in \dom(\bar W) \cap \dom(W)$, 
	$\bar W(\eta) = W(\eta)$. 
	If $\eta \in \dom(\bar W)$, 
	then $Y(\eta) = \bar W(\eta)$, and if $\eta \in \dom(W)$, then $Y(\eta) = W(\eta)$.

	Choosing any $\bar \delta < \delta$ in $C_h$ such that 
	$\bar \delta > \max(T)$, 
	it is simple to check that 
	$(\bar T,\bar W,\bar D)$ and $(T,W,D)$ are $(\bar \delta,\delta)$-split. 
	By Lemma 2.15, it follows that 
	$(\bar T,\bar W,\bar D) \oplus (T,W,D)$ is in $\p^*$ and extends 
	$(\bar T,\bar W,\bar D)$ and $(T,W,D)$.
	We know that $A \cap \sk(\theta) \subseteq \bar A \subseteq \sk(\theta)$. 
	By Theorem 5.15, $C = \bar A \cup A$ is adequate. 
	Also, obviously $\bar A$ and $A$ are subsets of $C$.

	It remains to prove that $Y$ is $C$-separated. 
	Let $M \in C$, let $\eta$ and $\xi$ be distinct elements of $M \cap \dom(Y)$, 
	and let $x \in Y(\eta) \cap Y(\xi)$. 
	We prove that $x \in M$.

	\emph{Case 1:} $M \in \bar{A}$ and $\eta$ and $\xi$ are both in $\dom(\bar{W})$. 
	Then $Y(\eta) = \bar{W}(\eta)$ and $Y(\xi) = \bar{W}(\xi)$. 
	So $x \in \bar{W}(\eta) \cap \bar{W}(\xi)$. 
	Since $\bar{W}$ is $\bar{A}$-separated, $x \in M$.

	\emph{Case 2:} $M \in \bar{A}$ and at least one of $\eta$ or $\xi$ is not in $\dom(\bar{W})$. 
	Without loss of generality, assume that $\eta \notin \dom(\bar{W})$. 
	Since $\dom(W) \cap \theta \subseteq \dom(\bar{W})$, 
	$\eta$ is not in $\theta$. 
	But $\eta \in M \subseteq \theta$, which is a contradiction.

	\emph{Case 3:} $M \in A$ and $\eta$ and $\xi$ are both in $\dom(W)$. 
	Then $Y(\eta) = W(\eta)$ and $Y(\xi) = W(\xi)$. 
	So $x \in W(\eta) \cap W(\xi)$, and therefore $x \in M$ 
	since $W$ is $A$-separated.

	\emph{Case 4:} $M \in A$ and neither $\eta$ nor $\xi$ are in $\dom(W)$. 
	Then $\eta$ and $\xi$ are in $\dom(\bar{W}) \subseteq \theta$. 
	Since $A$ is $\theta$-closed and $\theta^\omega \subseteq \sk(\theta)$,  
	$M \cap \theta \in A \cap \sk(\theta) \subseteq \bar{A}$. 
	So $\eta$ and $\xi$ are in $(M \cap \theta) \cap \dom(\bar W)$. 
	By Case 1, $x \in M \cap \theta \subseteq M$.

	\emph{Case 5:} $M \in A$ and one of $\eta$ or $\xi$ is in $\dom(W)$ 
	and the other is not in $\dom(W)$. 
	Without loss of generality, assume that 
	$\eta \in \dom(W)$ and $\xi \in \dom(\bar W) \setminus \dom(W)$. 
	Then $\xi \in \theta$, $Y(\eta) = W(\eta) = \bar W(f(\eta))$ by (3), 
	and $Y(\xi) = \bar W(\xi)$. 
	By (4), $M \cap \omega_1 = g(M) \cap \omega_1$ and $M \cap \theta \subseteq g(M)$. 
	So $\xi \in g(M)$. 
	Also, $\eta \in M$ implies that $f(\eta) \in g(M)$ by (5). 
	So $x \in \bar W(f(\eta)) \cap \bar W(\xi)$ and $g(M) \in \bar A$. 
	As $\bar{A}$ is $\bar{W}$-separated, it follows that 
	$x \in g(M) \cap \omega_1 \subseteq M$.
\end{proof}

\begin{proposition}
	Let $\theta \in \Sigma$. 
	Then $D_\theta$ is dense in $\p$. 
	In fact, if $q \in \p$ and $N \in A_q$, then there exists $r \le q$ 
	which is in $D_\theta$ and satisfies that $A_r$ is $N$-closed.
\end{proposition}

\begin{proof}
	Let $\mathcal B = (H(\kappa),\in,\psi,\mathcal X,\p)$. 
	Consider $q \in \p$ and $N \in A_q$, 
	and we find an extension $r \le q$ which is in $D_\theta$ and satisfies that 
	$A_r$ is $N$-closed. 
	Write $q = (T,W,D,A)$. 
	By extending further if necessary using in succession  
	Lemmas 6.7 and 6.8, we may assume that 
	$A$ is $N$-closed and $\theta$-closed.

	Enumerate $\dom(W) = \{ \eta_0,\ldots,\eta_{m-1} \}$ and 
	$A = \{ K_0,\ldots,K_{n-1} \}$, where $m, n < \omega$. 
	Define:
	\begin{itemize}
		\item $U_0 = \{ k < m : \eta_k \in \theta \}$;
		\item $U_1 = \{ l < n : M_l \in \sk(\theta) \}$;
		\item $U_2 = \{ (k,l) \in m \times n : \eta_k \in K_l \}$;
		\item $U_3 = \{ (j,k) \in m \times m : \{ \eta_j, \eta_k \} \in D \}$;
		\item $U_4 = \{ (k,l) \in n \times n : K_k \subseteq K_l \}$.
	\end{itemize}
	Define a formula with free variables 
	$$
	\varphi = \varphi(\dot q,\dot T,\dot W,\dot D,\dot A,\dot \eta_0,\ldots,\dot \eta_{m-1},
	\dot K_0,\ldots,\dot K_{n-1})
	$$
	to be the conjunction of the following statements:
	\begin{enumerate}
		\item $\dot q = (\dot T,\dot W,\dot D,\dot A) \in \p$;
		\item $\dot T = T$;
		\item $\dom(\dot W) = \{ \dot \eta_0,\ldots,\dot \eta_{m-1} \}$;
		\item $\dot A = \{ \dot K_0,\ldots,\dot K_{n-1} \}$;
		\item for all $k \in U_0$, $\dot \eta_k = \eta_k$;
		\item for all $l \in U_1$, $\dot K_l = K_l$;
		\item for all $k < m$, $\dot W(\dot \eta_k) = W(\eta_k)$;
		\item for all $l < n$, $\dot K_l \cap \omega_1 = K_l \cap \omega_1$ 
		and $K_l \cap \theta \subseteq \dot K_l$;
		\item for all $(k,l) \in m \times n$, $(k,l) \in U_2$ 
		iff $\dot \eta_k \in \dot K_l$;
		\item for all $(j,k) \in m \times m$, $(j,k) \in U_3$ iff 
		$\{ \dot \eta_j, \dot \eta_k \} \in \dot D$;
		\item for all $(k,l) \in n \times n$, $(k,l) \in U_4$ 
		iff $\dot K_k \subseteq \dot K_l$;
		\item for all $l < n$, $\dot K_l \cap (N \cap \theta) \in \dot A$.
	\end{enumerate}
	Note that all of the parameters appearing in $\varphi$ 
	are members of $\sk(\theta)$, and that 
	$$
	\mathcal B \models \varphi[q,T,W,D,A,\eta_0,\ldots,\eta_{m-1},K_0,\ldots,K_{n-1}].
	$$
	By elementarity, we can find in $\sk(\theta)$ objects 
	$\bar q $, $\bar T$, $\bar W$, $\bar D$, 
	$\bar A$, $\bar \eta_0,\ldots,\bar \eta_{m-1}$, and 
	$\bar K_0,\ldots,\bar K_{n-1}$ which satisfy the same. 
	By (2), $\bar T = T$. 
	By (3) and (5), $\dom(W) \cap \theta \subseteq \dom(\bar W)$, 
	and by (4) and (6), $A \cap \sk(\theta) \subseteq \bar A$. 
	By (12), for all $K \in \bar A$, $K \cap N = K \cap (N \cap \theta) \in \bar A$. 
	So $\bar A$ is $N$-closed.

	Define functions $f : \dom(W) \to \dom(\bar W)$ and 
	$g : A \to \bar A$ by letting 
	$f(\eta_k) = \bar \eta_k$ for all $k < m$ and 
	$g(K_l) = \bar K_l$ for all $l < n$. 
	Using the definition of $\varphi$, 
	it is routine to check that $q$, $\bar q$, $f$, and $g$ 
	satisfy all of the assumptions of Lemma 9.3. 
	Hence, $\bar q \oplus q$ is in $\p$ and extends $\bar q$ and $q$. 
	Since $A_{\bar q}$ and $A_q$ are both $N$-closed, so 
	is $A_{\bar q} \cup A_q = A_{\bar q \oplus q}$.

	We claim that $\bar q \oplus q$ is in $D_\theta$, which completes the proof. 
	Write $\bar q \oplus q = (T,Y,E,C)$. 
	The set $A$ is $\theta$-closed, and if $K \in \bar A$, then 
	$K \in \sk(\theta)$ so $K \cap \theta = K \in \bar A$. 
	Hence, $C = A \cup \bar A$ is $\theta$-closed. 
	For the functions described in Definition 9.2, we use 
	$f_0 = f \res (\dom(W) \setminus \theta)$ and $g_0 = g \res (A \setminus \sk(\theta))$.

	(3a) Let $\eta \in \dom(Y) \setminus \theta = \dom(W) \setminus \theta$. 
	Fix $k < m$ such that $\eta = \eta_k$. 
	Then $Y(\eta_k) = W(\eta_k)$. 
	By (7), $Y(f_0(\eta)) = \bar W(\bar \eta_k) = W(\eta_k) = Y(\eta_k)$.
	
	(3b) Let $K \in C \setminus \sk(\theta) = A \setminus \sk(\theta)$. 
	Fix $l < n$ such that $K = K_l$. 
	By (8), $g_0(K_l) \cap \omega_1 = \bar K_l \cap \omega_1 = K_l \cap \omega_1$ and 
	$K_l \cap \theta \subseteq \bar K_l = g_0(K_l)$.

	(3c) Let $\eta \in \dom(Y) \setminus \theta = \dom(W) \setminus \theta$ and 
	let $K \in C \setminus \sk(\theta) = A \setminus \sk(\theta)$. 
	Fix $k < m$ and $l < n$ such that $\eta = \eta_k$ and $K = K_l$. 
	Then by (9), $\eta_k \in K_l$ iff $(k,l) \in U_2$ iff 
	$\bar \eta_k \in \bar K_l$ iff $f_0(\eta_k) \in g_0(K_l)$.

	(3d) Let $\eta \in \dom(Y) \setminus \theta$ and let $\xi \in \dom(Y) \cap \theta$. 
	Then $\eta \in \dom(W)$. 
	Assume that $\{ \eta, \xi \} \in E$. 
	Since $E = \bar D \cup D$ and $\eta \notin \theta$, 
	$\{ \eta, \xi \} \in D$. 
	So $\xi \in \dom(W)$. 
	Fix $j, k < m$ such that $\eta = \eta_j$ and $\xi = \eta_k$. 
	Then $(j,k) \in U_3$. 
	By (10), $\{ f_0(\eta), f_0(\xi) \} \in \bar D$. 
	But $\xi \in \theta$ implies that $k \in U_0$. 
	Hence by (5), $f_0(\xi) = \xi$. 
	So $\{ f_0(\eta), \xi \} \in \bar D$, and therefore 
	$\{ f_0(\eta), \xi \} \in E$.
	
	(3e) Let $\eta, \xi \in \dom(Y) \setminus \theta$. 
	Then $\eta, \xi \in \dom(W)$. 
	Assume that $\{ \eta, \xi \} \in E$. 
	Since $E = \bar D \cup D$ and $\eta \notin \theta$, 
	$\{ \eta, \xi \} \in D$. 
	Fix $j, k < m$ such that $\eta = \eta_j$ and $\xi = \eta_k$. 
	Then $(j,k) \in U_3$. 
	By (10), $\{ f_0(\eta), f_0(\xi) \} \in \bar D \subseteq E$.
	
	(3f) Let $K, M \in C \setminus \sk(\theta)$ and assume that $K \subseteq M$. 
	Fix $k, l < n$ such that $K = K_k$ and $M = K_l$. 
	Then $(k,l) \in U_4$. 
	By (11), $f_0(K_k) = \bar K_k \subseteq \bar K_l = f_0(K_l)$.
\end{proof}

\section{Projection and Chain Condition}

In this section, we prove that for all $\theta \in \Sigma$, 
a certain natural map of a dense subset of $\p$ into 
the suborder $\p \cap \sk(\theta)$ is 
a projection mapping.

\begin{definition}
	For any $\theta \in \Sigma$, let $\p_\theta = \p \cap \sk(\theta)$.
\end{definition}

\begin{definition}
	For any $\theta \in \Sigma$, define $\pi_\theta$ with domain $\p$ by 
	letting 
	$$
	\pi_\theta(p) = (T_p,W_p \res \theta,D_p \cap [\theta]^2,A_p \cap \sk(\theta)).
	$$
\end{definition}

\begin{lemma}
	Let $\theta \in \Sigma$. 
	\begin{enumerate}
	\item For any $p \in \p$, 
	$\pi_\theta(p) \in \p_\theta$ and $p \le \pi_\theta(p)$.
	\item If $q \le p$, then $\pi_\theta(q) \le \pi_\theta(p)$.
	\item If $q \le s$, where $q \in \p$ and $s \in \p_\theta$, then 
	$\pi_\theta(q) \le s$.
	\end{enumerate}
\end{lemma}

The proof is straightforward.

\begin{lemma}
	Let $\theta \in \Sigma$. 
	Let $1 < d < \omega$ and suppose 
	that $p_0,\ldots,p_{d-1}$ are in $\p$ and 
	$p_0 \oplus \cdots \oplus p_{d-1} \in \p$. 
	Then 
	$$
	\pi_\theta(p_0 \oplus \cdots \oplus p_{d-1}) = 
	\pi_\theta(p_0) \oplus \cdots \oplus \pi_\theta(p_{d-1}).
	$$
\end{lemma}

The proof is easy.

\begin{definition}
	Define a function $w_\theta$ as follows. 
	The domain of $w_\theta$ is the set of ordered pairs $(q, N)$ such that 
	$q \in D_\theta$, $N \in A_q$, and $A_q$ is $N$-closed. 
	For any such ordered pair $(q,N)$, define 
	$$
	w_\theta(q,N) = 
	w(q,N)^{\frown} \langle f \res N, g \res \sk(N) \rangle,
	$$
	where $f$ and $g$ are the $\lhd$-least witnesses to the fact 
	that $q \in D_\theta$. 
\end{definition}

\begin{lemma}
	Let $\theta \in \Sigma$. 
	Suppose that $p$ and $q$ are in $D_\theta$ as witnessed by functions 
	$f_p$ and $g_p$ for $p$ and $f_q$ and $g_q$ for $q$. 
	Assume that $w_\theta(p,M) = w_\theta(q,N)$ and $p \in \sk(N)$. 
	Then:
	\begin{enumerate}
		\item For all $\eta \in \dom(W_q) \setminus \theta$ and for all 
		$K \in A_p \setminus \sk(\theta)$, 
		$\eta \in K$ implies that $f_q(\eta) \in g_p(K)$.
		\item For all $\eta \in \dom(W_p) \setminus \theta$ 
		and for all $K \in A_q \setminus \sk(\theta)$ such that $K < N$, 
		$\eta \in K$ implies that $f_p(\eta) \in g_q(K)$.	
	\end{enumerate}
\end{lemma}

\begin{proof}
	Since $w_\theta(p,M) = w_\theta(q,N)$, 
	$f_p \res M = f_q \res N$ and $g_p \res \sk(M) = g_q \res \sk(N)$. 
	By Lemma 8.7, for all $\eta \in (\dom(W_p) \cap \dom(W_q)) \setminus \theta$, 
	$\eta \in M \cap N$ and hence $f_p(\eta) = f_q(\eta)$, 
	and for all $K \in ((A_p \cap A_q) \setminus \sk(\theta))$, 
	$K \in \sk(M) \cap \sk(N)$ and so $g_p(K) = g_q(K)$.

	(1) Let $\eta \in \dom(W_q) \setminus \theta$ and  
	$K \in A_p \setminus \sk(\theta)$, and suppose that $\eta \in K$. 
	First, assume that $\eta \in \dom(W_p) \cap \dom(W_q)$. 
	Then $f_p(\eta) = f_q(\eta)$. 
	Since $\eta \in K$, by Definition 9.2(c), 
	$f_q(\eta) = f_p(\eta) \in g_p(K)$. 
	Now assume that $\eta \in \dom(W_q) \setminus \dom(W_p)$. 
	Since $p \in \sk(N)$, $K \in \sk(N)$. 
	As $\eta \in K$, $\eta \in \dom(W_q) \cap N = \dom(W_p) \cap M$, which 
	contradicts that $\eta \notin \dom(W_p)$.

	(2) Let $\eta \in \dom(W_p) \setminus \theta$ and 
	$K \in A_q \setminus \sk(\theta)$ with $K < N$, and suppose that $\eta \in K$. 
	First, assume that $\eta \in \dom(W_p) \cap \dom(W_q)$. 
	Then $f_p(\eta) = f_q(\eta)$. 
	Since $\eta \in K$, by Definition 9.2(c), 
	$f_p(\eta) = f_q(\eta) \in g_q(K)$. 
	Secondly, assume that $K \in A_p \cap A_q$. 
	Then $g_p(K) = g_q(K)$. 
	Since $\eta \in K$, by Definition 9.2(c), 
	$f_p(\eta) \in g_p(K) = g_q(K)$.

	Now assume that $\eta \in \dom(W_p) \setminus \dom(W_q)$ 
	and $K \in A_q \setminus A_p$. 
	Since $p \in \sk(N)$, $\eta \in N$. 
	So $\eta \in K \cap N$. 
	Since $(q,N) \in \dom(w)$, $A_q$ is $N$-closed. 
	As $K < N$, it follows that $K \cap N \in A_q \cap \sk(N) = A_p \cap \sk(M)$. 
	Since $\eta \in (K \cap N) \setminus \theta$, $K \cap N \notin \sk(\theta)$. 
	So $g_p(K \cap N) = g_q(K \cap N)$. 
	As $\eta \in K \cap N$, by Definition 9.2(c), 
	$f_p(\eta) \in g_p(K \cap N) = g_q(K \cap N)$. 
	By Definition 9.2(f), 
	$g_q(K \cap N) \subseteq g_q(K)$, so $f_p(\eta) \in g_q(K)$.	
\end{proof}

It follows from Proposition 10.9 below that $\pi_\theta$ restricted to the dense 
set $D_\theta$ is a projection mapping into $\p_\theta$. 
However, in order to prove that quotient forcings of $\p$ are Y-proper, this is not enough. 
In particular, Lemma 11.6 below needs $\pi_\theta$ to be a projection mapping on a larger 
set of conditions, which we introduce now.

\begin{definition}
	Let $\theta \in \Sigma$. 
	Define $E_\theta$ to be the set of conditions $p \in \p$ such that 
	either $p \in D_\theta$ and $A_p \ne \emptyset$, or else 
	for some $1 < d < \omega$, $p_0,\ldots,p_{d-1}$, and $N_0,\ldots,N_{d-1}$:
	\begin{enumerate}
		\item $p_0,\ldots,p_{d-1}$ are in $D_\theta$;
		\item $p = p_0 \otimes \cdots \otimes p_{d-1}$;
		\item for all $i < d$, $(p_i,N_i) \in \dom(w_\theta)$;
		\item for all $i < j < d$, 
		$w_\theta(p_i, N_i) = w_\theta(p_j,N_j)$;
		\item for all $i < j < d$, 
		$p_i \in \sk(N_j)$.
	\end{enumerate}
\end{definition}

Note that by Lemma 8.6, in the above $p$ is an extension of each of 
$p_0,\ldots,p_{d-1}$.

\begin{lemma}
	For any $\theta \in \Sigma$, $E_\theta$ is dense in $\p$.
\end{lemma}

\begin{proof}
	Immediate by Lemma 6.6 and Proposition 9.4. 
\end{proof}

\begin{proposition}
	Let $\theta \in \Sigma$. 
	Assume that $p \in E_\theta$ and $s \le \pi_\theta(p)$ in $\p_\theta$. 
	Define $Y$ with domain equal to $\dom(W_s) \cup \dom(W_p)$ so that for all 
	$\eta \in \dom(W_s)$, $Y(\eta) = W_s(\eta)$, and for all 
	$\xi \in \dom(W_p) \setminus \dom(W_s)$, 
	$Y(\xi)$ is the downward closure of $W_p(\xi)$ in $T_s$. 
	Then $(T_s,Y,D_s \cup D_p,A_s \cup A_p)$ is in $\p$ 
	and is an extension of $s$ and $p$.
\end{proposition}

\begin{proof}
	We begin by fixing some notation. 
	If $p \notin D_\theta$, then fix 
	$1 < d < \omega$, $p_0,\ldots,p_{d-1}$, and $N_0,\ldots,N_{d-1}$ 
	witnessing that $p \in E_\theta \setminus D_\theta$ 
	(and in particular, $p = p_0 \oplus \cdots \oplus p_{d-1}$), and 
	for each $i < d$, fix functions $f_i$ and $g_i$ witnessing that $p_i \in D_\theta$. 
	If $p \in D_\theta$, then let $p_0 = p$, let $N_0$ be any member of $A_p$, 
	and let $f_0$ and $g_0$ witness that $p \in D_\theta$. 	
	By Lemma 10.3(2), for all $i < d$, 
	$s \le \pi_\theta(p_i)$. 
	Since $T_{\pi_\theta(p)} = T_p$, $T_s$ is an end-extension of $T_p$. 
	For each $i < d$, write $p_i = (T_i,W_i,D_i,A_i)$.
	
	\textbf{Claim:} If $\eta \in \dom(W_p) \setminus \theta$ and $x \in Y(\eta)$, then 
	there exists $j < d$ such that $\eta \in \dom(W_j)$, 
	$x$ is in the downward closure of $W_j(f_j(\eta))$ 
	in $T_s$, and $x \in W_s(f_j(\eta))$.
	
	\emph{Proof:} Since $Y(\eta)$ is the downward closure of $W_p(\eta)$ in $T_s$, 
	we can fix 
	$j < d$ such that $x$ is in the downward closure of $W_j(\eta)$ in $T_s$. 
	But $W_j(\eta) = W_j(f_j(\eta))$ by Definition 9.2(a). 
	So $x$ is in the downward closure of $W_j(f_j(\eta))$ in $T_s$. 
	Since $s \le \pi_\theta(p_j)$, $W_j(f_j(\eta)) \subseteq W_s(f_j(\eta))$, 
	and as $W_s(f_j(\eta))$ is downward closed in $T_s$, 
	$x \in W_s(f_j(\eta))$.  This completes the proof of the claim.

	It is straightforward to check that 
	$(T_s,Y,D_s \cup D_p)$ is in $\p^*$ and extends 
	$(T_s,W_s,D_s)$ in $\p^*$. 
	Concerning $(T_s,Y,D_s \cup D_p)$ being an extension of $(T_p,W_p,D_p)$ in $\p^*$, 
	(a-c) of Definition 2.9 follow easily from the fact that $s \le \pi_\theta(p)$.
	
	For (d) of Definition 2.9, suppose that $\{ \eta, \xi \} \in D_p$ 
	and $x \in Y(\eta) \cap Y(\xi)$. 
	We show that there exists some $z \in W_p(\eta) \cap W_p(\xi)$ 
	such that $x \le_{T_s} z$. 
	Fix $i < d$ such that $\{ \eta, \xi \} \in D_i$.
	
	\emph{Case A:} $\eta$ and $\xi$ are both in $\theta$. 
	Since $s \le \pi_\theta(p)$, 
	easily $\eta$ and $\xi$ are both in $\dom(W_s)$. 
	Hence, $Y(\eta) = W_s(\eta)$ and $Y(\xi) = W_s(\xi)$. 
	So $x \in W_s(\eta) \cap W_s(\xi)$. 
	Also, $\{ \eta, \xi \} \in D_p \cap [\theta]^2 = D_{\pi_\theta(p)}$. 
	Since $s \le \pi_\theta(p)$, 
	there exists $z \in W_{\pi_\theta(p)}(\eta) \cap W_{\pi_\theta(p)}(\xi) = 
	W_p(\eta) \cap W_p(\xi)$ such that $x \le_{T_s} z$.

	\emph{Case B:} One of $\eta$ or $\xi$ is in $\theta$ and the other is not. 
	Without loss of generality, assume that $\eta \notin \theta$ and $\xi \in \theta$. 
	Since $s \le \pi_\theta(p)$, 
	easily $\xi \in \dom(W_s)$. 
	By Definition 9.2(d), $\{ f_i(\eta), \xi \} \in D_i$. 
	By the claim, fix $j < d$ such that $\eta \in \dom(W_j)$ and $x \in W_s(f_j(\eta))$. 
	Now $\eta \in \dom(W_i) \cap \dom(W_j)$ implies that $f_i(\eta) = f_j(\eta)$. 
	So $\{ f_j(\eta), \xi \} \in D_i \cap [\theta]^2 \subseteq D_p \cap [\theta]^2 
	= D_{\pi_\theta(p)}$. 
	As $s \le \pi_\theta(p)$ and $x \in W_{s}(f_j(\eta)) \cap W_s(\xi)$, 
	there exists $z \in W_{\pi_\theta(p)}(f_j(\eta)) \cap W_{\pi_\theta(p)}(\xi)$ 
	such that $x \le_{T_s} z$. 
	Then $z \in W_p(f_i(\eta)) \cap W_p(\xi)$. 
	Since $\{ f_i(\eta), \xi \} \in D_i$ and $p \le p_i$, there exists 
	$c \in W_i(f_i(\eta)) \cap W_i(\xi)$ such that $z \le_{T_p} c$. 
	Then $c \in W_i(\eta)$ by Definition 9.2(a). 
	So $c \in W_p(\eta) \cap W_p(\xi)$. 
	As $T_s$ end-extends $T_p$, 
	$x \le_{T_s} z \le_{T_s} c$ and so $x \le_{T_s} c$.

	\emph{Case C:} $\eta$ and $\xi$ are both in $\dom(W_p) \setminus \theta$. 
	By the claim, fix $j, k < d$ such that $\eta \in \dom(W_j)$, $\xi \in \dom(W_k)$, and 
	$x \in W_s(f_j(\eta)) \cap W_s(f_k(\xi))$. 
	By Definition 9.2(e), $\{ f_i(\eta), f_i(\xi) \} \in D_i$. 	
	Since $\eta \in \dom(f_i) \cap \dom(f_j)$ and 
	$\xi \in \dom(f_i) \cap \dom(f_k)$, 
	$f_i(\eta) = f_j(\eta)$ and $f_i(\xi) = f_k(\xi)$. 
	So $\{ f_j(\eta), f_k(\xi) \} \in D_i \cap [\theta]^2 \subseteq D_p \cap [\theta]^2 
	= D_{\pi_\theta(p)}$. 
	Since $s \le \pi_\theta(p)$ and $x \in W_s(f_j(\eta)) \cap W_s(f_k(\xi))$, 
	there exists some $z \in W_{\pi_\theta(p)}(f_j(\eta)) \cap W_{\pi_\theta(p)}(f_k(\xi))$ 
	such that $x \le_{T_s} z$. 
	Then $z \in W_p(f_i(\eta)) \cap W_p(f_i(\xi))$. 
	As $p \le p_i$ and $\{ f_i(\eta), f_i(\xi) \} \in D_i$, 
	there exists $c \in W_i(f_i(\eta)) \cap W_i(f_i(\xi))$ 
	such that $z \le_{T_p} c$. 
	Then $c \in W_i(\eta) \cap W_i(\xi)$ by Definition 9.2(a), 
	and hence $c \in W_p(\eta) \cap W_p(\xi)$. 
	As $T_s$ end-extends $T_p$, 
	$x \le_{T_s} z \le_{T_s} c$, so $x \le_{T_s} c$. 
	This completes the proof that 
	$(T_s,Y,D_s \cup D_p)$ is an extension of $(T_p,W_p,D_p)$ in $\p^*$.

	Since $A_{i}$ is $\theta$-closed for all $i < d$, it easily follows that 
	$A_p = A_0 \cup \cdots \cup A_{d-1}$ is $\theta$-closed. 
	As $s \le \pi_\theta(p)$, 
	$A_p \cap \sk(\theta) = A_{\pi_\theta(p)} \subseteq A_s$. 
	Also, $s \in \p_\theta$ implies $A_s \subseteq \sk(\theta)$. 
	By Theorem 5.15, $A_s \cup A_p$ is adequate.

	It remains to prove that $Y$ is $(A_s \cup A_p)$-separated. 
	Let $K \in A_s \cup A_p$, let $\eta$ and $\xi$ be distinct elements of $K \cap \dom(Y)$, 
	and let $x \in Y(\eta) \cap Y(\xi)$. 
	We prove that $x \in K$.

	\emph{Case 1:} $K \in A_s$ and $\eta$ and $\xi$ are both in $\dom(W_s)$. 
	Then by definition, $Y(\eta) = W_s(\eta)$ and $Y(\xi) = W_s(\xi)$. 
	Hence, $x \in W_s(\eta) \cap W_s(\xi)$. 
	As $W_s$ is $A_s$-separated, it follows that $x \in K$.

	\emph{Case 2:} $K \in A_s$ and at least one of $\eta$ or $\xi$ is not in $\dom(W_s)$. 
	Without loss of generality, assume that $\eta \notin \dom(W_s)$. 
	Since $\dom(W_p) \cap \theta \subseteq \dom(W_s)$, 
	$\eta$ is not in $\theta$. 
	But $\eta \in K \subseteq \theta$, which is a contradiction.

	\emph{Case 3:} $K \in A_p \setminus A_s$ and $\eta$ and $\xi$ are both in 
	$\dom(W_s)$. 
	Then $K \notin \sk(\theta)$. 
	We have that $Y(\eta) = W_s(\eta)$ and $Y(\xi) = W_s(\xi)$, 
	so $x \in W_s(\eta) \cap W_s(\xi)$. 
	Fix $i < d$ such that $K \in A_i$. 
	Then $\eta, \xi \in K \cap \theta \subseteq g_i(K)$ by Definition 9.2(b). 
	Since $W_s$ is $A_s$-separated and $g_i(K) \in A_s$, 
	it follows that $x \in g_i(K) \cap \omega_1 \subseteq K$.
	
	\emph{Case 4:} $K \in A_p \setminus A_s$ and one of $\eta$ or $\xi$ is in 
	$\dom(W_p) \setminus \theta$ and the other is in $\dom(W_s)$. 
	Without loss of generality, assume that $\eta \in \dom(W_p) \setminus \theta$ 
	and $\xi \in \dom(W_s)$. 
	Then $Y(\xi) = W_s(\xi)$. 
	By the claim, fix $j < d$ such that $\eta \in \dom(W_j)$, 
	$x$ is in the downward closure of $W_j(f_j(\eta))$ in $T_s$, and $x \in W_s(f_j(\eta))$.	
	Fix $i < d$ such that $K \in A_i \setminus \sk(\theta)$. 
	By Definition 9.2(b), $\xi \in K \cap \theta \subseteq g_i(K)$.

	First, consider the case that $j < i$ and $N_i \le K$. 
	Since $x$ is in the downward closure of $W_j(f(\eta))$ in $T_s$, 
	fix $z \in W_j(f(\eta))$ such that $x \le_{T_s} z$. 
	Then $j < i$ implies that $p_j \in \sk(N_i)$, 
	and hence $z \in N_i \cap \omega_1 \le K \cap \omega_1$. 
	So $z \in K \cap \omega_1$, and as $x \le z$, 
	$x \in K \cap \omega_1$ as well.

	Secondly, assume that either $i \le j$, or $j < i$ and $K < N_i$. 
	We claim that $f_j(\eta) \in g_i(K)$. 
	If $i = j$, then $\eta \in K$ implies by Definition 9.2(c) 
	that $f_j(\eta) \in g_j(K) = g_i(K)$. 
	If $i < j$, then $f_j(\eta) \in g_i(K)$ by Lemma 10.6(1). 
	If $j < i$ and $K < N_i$, then $f_j(\eta) \in g_i(K)$ by Lemma 10.6(2). 
	So indeed $f_j(\eta) \in g_i(K)$. 
	But $x \in W_s(\xi) \cap W_s(f_j(\eta))$ 
	and $g_i(K) \in A_i \cap \sk(\theta) = A_{\pi_\theta(p_i)} \subseteq A_s$. 
	Since $\xi$ and $f_j(\eta)$ are in $g_i(K)$ 
	and $W_s$ is $A_s$-separated, it follows that 
	$x \in g_i(K) \cap \omega_1 = K \cap \omega_1$.

	\emph{Case 5:} $K \in A_p \setminus A_s$ and 
	$\eta$ and $\xi$ are both in $\dom(W_p) \setminus \theta$. 	
	By the claim, fix $j < d$ such that $\eta \in \dom(W_j)$, 
	$x$ is in the downward closure of $W_j(f_j(\eta))$ in $T_s$, 
	and $x \in W_s(f_j(\eta))$.	
	And fix $k < d$ such that $\xi \in \dom(W_k)$, 
	$x$ is in the downward closure of $W_k(f_k(\xi))$ in $T_s$, 
	and $x \in W_s(f_k(\xi))$.	
	Fix $i < d$ such that $K \in A_i$.

	Without loss of generality, assume that $j \le k$. 
	If $j < i$ and $N_i \le K$, then by the same argument 
	as in the first subcase of Case 4, $x \in K$. 
	The remaining cases are: $i < j$, $j = i = k$, 
	$j = i < k$, $j < i < k$ and $K < N_i$, 
	$j < k = i$ and $K < N_i$, and $k < i$ and $K < N_i$. 
	Note that by Definition 9.2(c), 
	if $i = j$ then $f_j(\eta) \in g_i(K)$, and if $i = k$, then $f_k(\xi) \in g_i(K)$. 
	Combining this information with Lemma 10.6, it is routine to check that 
	both $f_j(\eta)$ and $f_k(\xi)$ are in $g_i(K)$. 
	But $g_i(K) \in A_s$ and $x \in W_s(f_j(\eta)) \cap W_s(f_k(\xi))$. 
	Since $W_s$ is $A_s$-separated, it follows that 
	$x \in g_i(K) \cap \omega_1 = K \cap \omega_1$.
\end{proof}

\begin{proposition}
	Let $\theta \in \Sigma$. 
	Then $\pi_\theta \res E_\theta : E_\theta \to \p_\theta$ 
	is a projection mapping.
\end{proposition}

\begin{proof}
	Clearly, $\pi_\theta$ maps the maximum condition 
	in $D_\theta$ to the maximum condition in $\p_\theta$. 
	The map $\pi_\theta$ is order-preserving by Lemma 10.3(2).

	Suppose that $p \in E_\theta$ and $s \le \pi_\theta(p)$ in $\p_\theta$. 
	We find $r \le p$ in $E_\theta$ such that $\pi_\theta(r) \le s$. 
	By extending further if necessary, 
	assume that $A_s$ is non-empty. 
	Define $Y$ with domain equal to $\dom(W_s) \cup \dom(W_p)$ so that for all 
	$\eta \in \dom(W_s)$, $Y(\eta) = W_s(\eta)$, and 
	for all $\xi \in \dom(W_p) \setminus \dom(W_s)$, 
	$Y(\xi)$ is the downward closure of $W_p(\xi)$ in $T_s$. 
	By Proposition 10.9, $(T_s,Y,D_s \cup D_p,A_s \cup A_p)$ 
	is in $\p$ and extends $p$ and $s$. 
	Since $D_\theta$ is dense, fix 
	$r \le (T_s,Y,D_s \cup D_p,A_s \cup A_p)$ in $D_\theta$. 
	As $A_s$ is non-empty, so is $A_r$, so $r \in E_\theta$. 
	Then $r \le p$, and since $r \le s$, $\pi_\theta(r) \le s$ by Lemma 10.3(3).
\end{proof}

\begin{proposition}
	The forcing poset $\p$ is $\kappa$-c.c.
\end{proposition}

\begin{proof}
	Let $A$ be a maximal antichain of $\p$, and we show that $|A| < \kappa$. 
	Fix a regular cardinal $\chi > \kappa$ such that $A \in H(\chi)$. 
	Fix an elementary substructure $Q$ of 
	$\mathcal B = (H(\chi),\in,\psi,\mathcal X,\p)$ satisfying that 
	$|Q| < \kappa$, $\theta = Q \cap \kappa \in \Sigma$, and $A \in Q$. 
	Note that by elementarity, $Q \cap H(\kappa) = \psi[\theta] = \sk(\theta)$.

	Suppose for a contradiction that $|A| \ge \kappa$. 
	Then in particular, $A$ is not a subset of $Q$, so we can fix some 
	$p \in A \setminus Q$. 
	Fix $q \le p$ in $E_\theta$. 
	Then $\pi_\theta(q) \in \sk(\theta) \subseteq Q$. 
	Since $A$ is maximal, by the elementarity of $Q$ we can fix some 
	$u \in Q \cap \p$ which extends both $\pi_\theta(q)$ 
	and some element $s$ of $Q \cap A$. 
	Then $u \in Q \cap H(\kappa) = \sk(\theta)$, so $u \in \p_\theta$. 
	Since $\pi_\theta \res E_\theta$ is a projection mapping, 
	fix $r \le q$ in $E_\theta$ such that 
	$\pi_\theta(r) \le u$. 
	By Lemma 10.3(1), $r \le \pi_\theta(r) \le u \le s$. 
	So $r$ is below both $p$ and $s$. 
	Since $A$ is an antichain, $p = s$, which is false since $s \in Q$ 
	and $p \notin Q$.
\end{proof}

\begin{corollary}
	The forcing poset $\p$ forces that $\omega_1^V = \omega_1$ and 
	$\kappa = \omega_2$.
\end{corollary}

\begin{proof}
	By Corollary 7.3, Proposition 7.4, and Proposition 10.11.
\end{proof}

Combining Corollary 10.12 with Propositions 7.6 and 7.7, we have the following 
theorem, where $\dot G$ is the canonical $\p$-name for a generic filter on $\p$.

\begin{thm}
	The forcing poset $\p$ forces that $T_{\dot G}$ is a normal infinitely splitting 
	Aronszajn tree which is strongly non-saturated.
\end{thm}

\section{The Quotient Forcing}

With a projection mapping at hand, we are now in a position to analyze 
the quotient forcing in an intermediate extension. 
In this section, we provide some information about the quotient forcing which we use in 
the next section to show that it is Y-proper on a stationary set.

For the remainder of the section, fix $\theta \in \Sigma$. 
Recall that $\pi_\theta \res E_\theta$ is a projection mapping 
from $E_\theta$ into $\p_\theta = \p \cap \sk(\theta)$. 
For any generic filter $H$ on $\p_\theta$, define in $V[H]$ 
the quotient forcings 
$E_\theta / H = \{ q \in E_\theta : \pi_\theta(q) \in H \}$ and 
$$
\p / H = \{ p \in \p : \exists q \in E_\theta / H  \ (q \le p) \},
$$
considered as suborders of $\p$. 

\begin{lemma}
	Let $H$ be a generic filter on $\p_\theta$. 
	\begin{enumerate}
		\item If $q \in \p / H$, $p \in \p$, and $q \le p$, then $p \in \p / H$.
		\item $\p / H$ is the set of all $q \in \p$ which are compatible in $\p$ 
		with every member of $H$.
		\item If $q \in \p / H$ and $s \in H$, then there exists some 
		$r \in \p / H$ which extends $q$ and $s$.
		\item If in $V$, $\mathcal D$ is a dense open subset of $\p$, then in $V[H]$, 
		$\mathcal D \cap (\p / H)$ is a dense open subset of $\p / H$.
	\end{enumerate}
\end{lemma}

\begin{proof}
	(1) is immediate. 
	(2), (3), and (4) have routine proofs using density arguments in $V$.
\end{proof}

The following lemma is standard.

\begin{lemma}
	\begin{enumerate}
		\item If $G$ is a $V$-generic filter on $\p$, 
		then $H = \pi_\theta[G] = G \cap \p_\theta$ 
		is a $V$-generic filter on $\p_\theta$, 
		$G$ is a $V[H]$-generic filter on $\p / H$, and $V[G] = V[H][G]$.
		\item If $H$ is a generic filter on $\p_\theta$ and 
		$G$ is a $V[H]$-generic filter on $\p / H$, then 
		$G$ is a $V$-generic filter on $\p$, $H = G \cap \p_\theta = \pi_\theta[G]$, 
		and $V[G] = V[H][G]$.
	\end{enumerate}
\end{lemma}

We occasionally write $\dot H$ for the canonical $\p_\theta$-name 
for a generic filter on $\p_\theta$ when working in $V$.

For the remainder of the section, fix a generic filter $H$ on $\p_\theta$.

\begin{lemma}
	If $p \in \p / H$, then $\pi_\theta(p) \in H$. 
	Consequently, $E_\theta / H = E_\theta \cap (\p / H)$.
\end{lemma}

The proof is easy.

The converse of Lemma 11.3 is false in general. 
For example, if $\{ M, N \} \subseteq \mathcal X$ is not adequate, 
$(\emptyset,\emptyset,\emptyset,\{ M \}) \in H$, 
and $N \notin \sk(\theta)$, then 
$\pi_\theta(\emptyset,\emptyset,\emptyset,\{ N \}) = 
(\emptyset,\emptyset,\emptyset,\emptyset) \in H$, 
but $(\emptyset,\emptyset,\emptyset,\{ M \})$ is incompatible with 
$(\emptyset,\emptyset,\emptyset,\{ N \})$. 
Therefore, the condition $(\emptyset,\emptyset,\emptyset,\{ N \})$ 
is not in $\p / H$ by Lemma 11.1(2).

\begin{definition}
Define $(T_H,<_H)$ by:
	\begin{itemize}
		\item $x \in T_H$ if there exists some $p \in H$ such that $x \in T_p$;
		\item $x <_{H} y$ if there exists some $p \in H$ such that $x <_p y$.
	\end{itemize}
\end{definition}

As usual, we abbreviate $(T_H,<_H)$ by $T_H$.

\begin{lemma}
	If $G$ is a $V[H]$-generic filter on $\p / H$, then $T_G = T_H$.
\end{lemma}

\begin{proof}
	This follows easily from Lemma 11.2 and the fact that for any $p \in \p$, 
	$T_p = T_{\pi_\theta(p)}$.
\end{proof}

\begin{lemma}
	Let $p \in E_\theta$. 
	Assume that $p_0,\ldots,p_{d-1}$ are in $\p / H$, where $1 < d < \omega$, 
	$p = p_0 \oplus \cdots \oplus p_{d-1}$, and $p$ extends 
	each of $p_0,\ldots,p_{d-1}$. 
	Suppose that 
	for all $i < j < d$, for all $x \in T_{p_i} \setminus T_{p_j}$ 
	and for all $y \in T_{p_j} \setminus T_{p_i}$, 
	$x$ and $y$ are incomparable in $T_H$. 
	Then $p \in \p / H$.
\end{lemma}

\begin{proof}
	The condition $\pi_\theta(p_0 \oplus \cdots \oplus p_{d-1})$ extends 
	each of $\pi_\theta(p_0),\ldots,\pi_\theta(p_{d-1})$. 
	By Lemma 10.4, $\pi_\theta(p) = \pi_\theta(p_0 \oplus \cdots \oplus p_{d-1}) = 
	\pi_\theta(p_0) \oplus \cdots \oplus \pi_\theta(p_{d-1})$. 
	It suffices to show that 
	$\pi_\theta(p_0) \oplus \cdots \oplus \pi_\theta(p_{d-1}) \in H$, 
	for then $\pi_\theta(p) \in H$ so $p \in E_\theta / H \subseteq \p / H$. 
	As $p_0,\ldots,p_{d-1}$ are all in $\p / H$, 
	for all $i < d$, $\pi_\theta(p_i) \in H$. 
	Fix $r \in H$ 
	such that $r \le \pi_\theta(p_i)$ for all $i < d$. 

	For all $i < d$, $T_{p_i} = T_{\pi_\theta(p_i)}$. 
	So for all $i < j < d$, for all 
	$x \in T_{\pi_\theta(p_i)} \setminus T_{\pi_\theta(p_j)}$ and for all 
	$y \in T_{\pi_\theta(p_j)} \setminus T_{\pi_\theta(p_i)}$, 
	$x$ and $y$ are incomparable in $T_H$. 
	Since $r \in H$, all such $x$ and $y$ are incomparable in $T_r$ as well. 
	Applying Lemma 6.10 to the conditions 
	$\pi_\theta(p_0),\ldots,\pi_\theta(p_{d-1})$ and $r$, 
	we get that $r \le \pi_\theta(p_0) \oplus \cdots \oplus \pi_\theta(p_{d-1})$.
	Since $r \in H$, $\pi_\theta(p_0) \oplus \cdots \oplus \pi_\theta(p_{d-1}) \in H$.
\end{proof}

\begin{definition}
	Define $\mathcal X(H)$ to be the set of all $N \in \mathcal X$ 
	such that $(\emptyset,\emptyset,\emptyset,\{N \cap \theta\}) \in H$.
\end{definition}

\begin{lemma}
	In $V[H]$, $\mathcal X(H)$ is a stationary subset of $[\kappa]^\omega$.
\end{lemma}

\begin{proof}
	We give a density argument in $V$. 
	Let $p \in \p_\theta$ and suppose that $\dot F$ is a $\p_\theta$-name 
	for a function from $\kappa^{<\omega}$ to $\kappa$. 
	We find $s \le p$ in $\p_\theta$ and $N \in \mathcal X$ such that 
	$s$ forces that $N$ is closed under $\dot F$ and $N \in \mathcal X(\dot H)$.

	Fix a regular cardinal $\chi > \kappa$ with $\dot F \in H(\chi)$. 
	Let $M$ be a countable elementary substructure of 
	$(H(\chi),\in,\psi,\p,\theta,\dot H,\dot F)$ 
	such that $p \in M$ and $N = M \cap \kappa \in \mathcal X$. 
	By Theorem 7.2, $p + N$ is in $\p$, $p + N \le p$, and 
	$p + N$ is $(M,\p)$-generic. 
	Since $p + N$ is $(M,\p)$-generic, $p + N$ forces that 
	$M$, and hence $N$, is closed under $\dot F$. 

	Fix $r \le p + N$ in $D_\theta$. 
	Since $r \in D_\theta$ and $N \in A_r$, $N \cap \theta \in A_r$. 
	As $\theta^\omega \subseteq \sk(\theta)$, $N \cap \theta \in \sk(\theta)$. 
	So $N \cap \theta \in A_r \cap \theta = A_{\pi_\theta(r)}$. 
	Therefore, $\pi_{\theta}(r) \le 
	(\emptyset,\emptyset,\emptyset,\{ N \cap \theta \})$ in $\p_\theta$. 
	So $\pi_\theta(r)$ forces that $N \in \mathcal X(\dot H)$. 
	Since $r \le p$ and $p \in \p_\theta$, 
	by Lemma 10.3(3), $\pi_\theta(r) \le p$. 
	As $r$ forces in $\p$ that $N$ is closed under $\dot F$, an easy argument using 
	Lemma 11.2 shows that $\pi_\theta(r)$ forces in $\p_\theta$ 
	that $N$ is closed under $\dot F$.
\end{proof}

\begin{lemma}
	Suppose that $p \in D_\theta \cap (\p / H)$, $A_p$ is non-empty, 
	$N \in \mathcal X(H)$, 
	$p$ and $\theta$ are in $\sk(N)$, and $A_p$ is non-empty. 
	Then $p + N$ is in $\p / H$ and is an extension of $p$.
\end{lemma}

\begin{proof}
	By Lemma 6.6, $p + N$ is in $\p$ and is an extension of $p$. 
	We prove that $p + N \in \p / H$ by giving a density argument in $V$. 
	Assume that $s \in \p_\theta$ 
	and $s$ forces that $N \in \mathcal X(\dot H)$ and $p \in \p / \dot H$. 
	We find an extension of $s$ in $\p_\theta$ which forces 
	that $p + N \in \p / \dot H$. 
	By extending further if necessary using Lemma 11.3, we may assume that 
	$s \le \pi_\theta(p)$ and $N \cap \theta \in A_s$. 
	To show that $s$ forces that $p + N \in \p / \dot H$, 
	by Lemma 11.1(2) it suffices to show 
	that for all $t \le s$ in $\p_\theta$, $t$ and $p + N$ are compatible in $\p$.

	Fix $t \le s$ in $\p_\theta$. 
	Define $Y$ with domain equal to $\dom(W_t) \cup \dom(W_p)$ so that for all 
	$\eta \in \dom(W_t)$, $Y(\eta) = W_t(\eta)$, and for all 
	$\xi \in \dom(W_p) \setminus \dom(W_t)$, 
	$Y(\xi)$ is the downward closure of $W_p(\xi)$ in $T_t$. 
	Since $p \in E_\theta$ and $t \le \pi_\theta(p)$, by Proposition 10.9 we have that 
	$u = (T_t,Y,D_t \cup D_p,A_t \cup A_p)$ is in $\p$ and extends $t$ and $p$. 
	Define $v = u + N$.

	We prove that $v \in \p$ and $v$ is an extension of $t$ and $p + N$, 
	which completes the proof. 
	Now $u$ is in $\p$, $u$ extends $t$ and $p$, and 
	$(T_u,W_u,D_u) = (T_v,W_v,D_v)$. 
	It follows that
	$(T_v,W_v,D_v)$ is in $\p^*$ and extends $(T_t,W_t,D_t)$ and 
	$(T_p,W_p,D_p)$ in $\p^*$. 
	Now $A_t$ and $A_p$ are subsets of $A_u$, and since $A_v = A_u \cup \{ N \}$, clearly 
	$A_t$ and $A_{p+N} = A_p \cup \{ N \}$ are subsets of $A_v$. 
	We have that $A_v = A_u \cup \{ N \} = A_t \cup A_p \cup \{ N \}$, 
	$A_u = A_t \cup A_p$ is adequate, and $A_{p+N} = A_p \cup \{ N \}$ is adequate. 
	So $A_v$ is adequate provided that $A_t \cup \{ N \}$ is adequate. 
	Since $A_t \in \sk(\theta)$ and $N \cap \theta \in A_t$, 
	$A_t \cup \{ N \}$ is adequate by Lemma 5.16.

	Finally, we show that $W_v$ is $A_v$-separated. 
	We know that $W_v = Y$ and $Y$ is $A_u$-separated. 
	So it suffices to show that if $\eta$ and $\xi$ are distinct elements of $N \cap \dom(Y)$ 
	and $x \in Y(\eta) \cap Y(\xi)$, then $x \in N$. 
	First, assume that $\eta$ and $\xi$ are in $\theta$. 
	Then $\eta$ and $\xi$ are in $N \cap \theta$. 
	Since $t \le \pi_\theta(p)$, 
	$\dom(W_p) \cap \theta \subseteq \dom(W_t)$, and hence 
	$Y(\eta) = W_t(\eta)$ and $Y(\xi) = W_t(\xi)$. 
	So $x \in W_t(\eta) \cap W_t(\xi)$ and $N \cap \theta \in A_t$, 
	Since $W_t$ is $A_t$-separated, $x \in N \cap \theta \subseteq N$.

	For the remaining cases, fix the $\lhd$-least functions 
	$f$ and $g$ witnessing that $p \in D_\theta$. 
	Since $p$ and $\theta$ are members of $\sk(N)$, 
	it is clear by elementarity that $f$ and $g$ are also in $\sk(N)$.
	
	\textbf{Claim:} If $\eta \in \dom(W_p) \setminus \theta$ and $x \in Y(\eta)$, 
	then $x \in W_t(f(\eta))$. 

	\emph{Proof:} We have that $Y(\eta)$ is the downward closure of $W_p(\eta)$ 
	in $T_t$. 
	By Definition 9.2(a), $W_p(\eta) = W_p(f(\eta))$. 
	So $x$ is in the downward closure of $W_p(f(\eta))$ in $T_t$. 
	As $t \le \pi_\theta(p)$, $W_p(f(\eta)) = W_{\pi_\theta(p)}(f(\eta)) \subseteq W_s(f(\eta))$. 
	Since $W_s(f(\eta))$ is downwards closed in $T_s$, $x \in W_s(f(\eta))$. 
	This completes the proof of the claim.
	
	Assume that one of $\eta$ or $\xi$ is in $\theta$ and the other is not. 
	Without loss of generality, assume that $\eta \in \theta$ and $\xi \notin \theta$. 
	Since $f$ and $\xi$ are in $\sk(N)$, $f(\xi) \in N$. 
	By the claim, $x \in W_t(f(\xi))$. 
	So we have that $\eta$ and $f(\xi)$ are in $N \cap \theta$, $N \cap \theta \in A_t$, 
	and $x \in W_t(\eta) \cap W_t(f(\xi))$. 
	Since $W_t$ is $A_t$-separated, $x \in N \cap \theta \subseteq N$. 
	Finally, assume that $\eta$ and $\xi$ are both not in $\theta$. 
	By the claim, $x \in W_s(f(\eta)) \cap W_s(f(\xi))$. 
	Since $f$, $\eta$, and $\xi$ are in $\sk(N)$, $f(\eta)$ and $f(\xi)$ are in $N \cap \theta$. 
	As $N \cap \theta \in A_t$ and $W_t$ is $A_t$-separated, it follows that 
	$x \in N \cap \theta \subseteq N$.
\end{proof}

\section{Quotients Are Indestructibly Y-Proper}

We are finally ready to prove that quotient forcings of $\p$ are Y-proper 
on a stationary set.
 
\begin{thm}
	Let $\theta \in \Sigma$. 
	Suppose that $H$ is a generic filter on $\p_\theta$. 
	Assume that $W$ is a transitive model of \textsf{ZFC} with the same ordinals as $V$ 
	satisfying:
	\begin{itemize}
		\item $V[H] \subseteq W$;
		\item $\omega_1^V = \omega_1^W$;
		\item $T_H$ is an Aronszajn tree in $W$;
		\item $\kappa$ is a regular cardinal in $W$;
		\item $\mathcal X(H)$ is a stationary subset of $[\kappa]^\omega$ in $W$.
	\end{itemize}
	Then in $W$, $\p / H$ is Y-proper on a stationary set.	
\end{thm}

For our purposes, we are primarily interested in the special cases that either $W = V[H]$, 
or $W$ is a generic extension of $V[H]$ by a Y-proper forcing. 
Recall that Y-proper forcings do not add new cofinal branches of $\omega_1$-trees  
and are proper, and therefore preserve stationary subsets of $[\kappa]^\omega$.

\begin{corollary}
	Let $\theta \in \Sigma$. 
	Suppose that $H$ is a generic filter on $\p_\theta$. 
	Then in $V[H]$, $\p / H$ is Y-proper on a stationary set.
\end{corollary}

\begin{corollary}
	Let $\theta \in \Sigma$, let $H$ be a generic filter on $\p_\theta$, 
	and let $\q$ be a Y-proper forcing in $V[H]$. 
	Then $\q$ forces over $V[H]$ that $\p / H$ is Y-proper on a stationary set.
\end{corollary}

For the remainder of the section, fix $\theta$, $H$, and $W$ as in the 
statement of Theorem 12.1. 
All of the results in this section are intended to take place in $W$. 
A key point in what follows 
is that many of the properties related to the compatibility of 
conditions in $\p$ or $\p / H$ are absolute between $V$ or $V[H]$ and $W$.

\begin{thm}
	In $W$, let $\chi > \kappa$ be regular and let $M$ be a countable elementary 
	substructure of $\mathcal B = (H(\chi),\in,\psi,\p,\theta,H,\p / H,D_\theta,w_\theta)$ 
	such that $N = M \cap \kappa \in \mathcal X(H)$. 
	Then for any $u \in M \cap D_\theta \cap (\p / H)$ such that $A_u$ is non-empty, 
	$u + N$ is in $\p / H$, $u + N \le u$, 
	and $u + N$ is $(M,\p / H)$-generic.
\end{thm}

\begin{proof}
	Throughout the proof we work in $W$. 
	The short proof of Lemma 5.7 is easily adjusted to show that 
	$M \cap H(\kappa)^V = \sk(N)$. 
	By Lemma 11.9, $u + N$ is in $\p / H$ and extends $u$. 
	To show that $u + N$ is $(M,\p / H)$-generic, fix $q \le u$ in $\p / H$ and 
	fix $\mathcal{D} \in M$ which is a dense open subset of $\p / H$. 
	By extending further if necessary using Proposition 9.4 and Lemma 11.1(4), 
	we may assume that 
	$q$ is in $\mathcal D \cap D_\theta$ and $A_q$ is $N$-closed. 
	Then $(q,N) \in \dom(w_\theta)$. 
	Let $\vec z = w_\theta(q,N)$.

	We claim that there exists some $(\bar q,\bar N) \in M$ such that 
	$\bar q \in \mathcal D \cap D_\theta$, 
	$w_\theta(\bar q,\bar N) = \vec z$, and for all $x \in T_q \setminus T_{\bar q}$ and 
	for all $y \in T_{\bar q} \setminus T_q$, $x$ and $y$ are 
	incomparable in $T_H$. 
	Suppose not. 	
	Let $\mathcal I$ be the set of all ordered pairs $(\bar q,\bar N)$ such that 
	$\bar q \in \mathcal D \cap D_\theta$ and $w_\theta(\bar q,\bar N) = \vec z$. 
	Note that $(q,N) \in \mathcal I$ and $\mathcal I \in M$ by elementarity. 

	Define in $M$ by induction a sequence 
	$\langle (q_\alpha,N_\alpha) : \alpha < \omega_1 \rangle$ 
	of members of $\mathcal I$ so that for all $\alpha < \beta < \omega_1$, 
	$(q_\alpha,N_\alpha) \in \sk(N_\beta)$ and there exist 
	$x \in T_{q_\alpha} \setminus T_{q_\beta}$ and 
	$y \in T_{q_\beta} \setminus T_{q_\alpha}$ such that 
	$x$ and $y$ are comparable in $T_H$. 
	If the induction fails, then there exists $\delta \in M \cap \omega_1$ and 
	$\langle (q_\alpha,N_\alpha) : \alpha < \delta \rangle \in M$ 
	satisfying the required properties, but this sequence 
	cannot be extended any further. 
	But then $(q,N)$ is a witness that this sequence can be extended further, 
	which is a contradiction. 

	Let $q_\alpha = (T_\alpha,W_\alpha,D_\alpha,A_\alpha)$ and 
	$\delta_\alpha = N_\alpha \cap \omega_1$ for all $\alpha < \omega_1$. 	
	By Lemma 8.7, for all $\alpha < \beta < \omega_1$, 
	$T_\alpha \cap T_\beta = T_\alpha \cap \delta_\alpha = T_\beta \cap \delta_\beta$. 
	Hence, $\{ T_\alpha \setminus \delta_\alpha : \alpha < \omega_1 \}$ is a disjoint 
	family of finite subsets of the Aronszajn tree $T_H$. 
	By the theorem of Baumgartner-Malitz-Reinhardt stated at the end of the introduction, 
	there exist $\alpha < \beta < \omega_1$ such that every element of 
	$T_\alpha \setminus \delta_\alpha$ is incomparable in $T_H$ with every element 
	of $T_\beta \setminus \delta_\beta$. 
	But $T_\alpha \setminus \delta_\alpha = T_\alpha \setminus T_\beta$ 
	and $T_\beta \setminus \delta_\beta = T_\beta \setminus T_\alpha$, 
	and we have a contradiction to the definition of the sequence.

	So indeed, there exists some $(\bar q,\bar N) \in M$ such that 
	$\bar q \in \mathcal D \cap D_\theta$, $w_\theta(\bar q,\bar N) = \vec z$, 
	and for all $x \in T_q \setminus T_{\bar q}$ and 
	for all $y \in T_{\bar q} \setminus T_q$, $x$ and $y$ are 
	incomparable in $T_H$. 
	Then $(\bar q,\bar N) \in M \cap H(\kappa)^V = \sk(N)$. 
	By Lemma 8.6, $\bar q \oplus q$ a condition in $\p$ which extends $\bar q$ and $q$. 
	By Lemma 11.6, $\bar q \oplus q$ is in $\p / H$.
\end{proof}

\begin{definition}
	Let $\vec z$ be in the range of $w_\theta$. 
	A set $R \subseteq D_\theta \cap (\p / H)$ is said to be \emph{$\vec z$-robust} 
	if the set 
	$$
	\{ N \in \mathcal X(H) : \exists q \in R \ (w_\theta(q, N) = \vec z) \}
	$$
	is stationary in $[\kappa]^\omega$.
\end{definition}

\begin{proposition}
	For any $\vec z$ in the range of $w_\theta$, the collection 
	$$
	\left\{ \ \sum R : 
	R \subseteq D_\theta \cap (\p / H) \ \text{is $\vec z$-robust} \ \right\}
	$$
	is a centered subset of $\mathcal B(\p / H)$.
\end{proposition}

\begin{proof}
	Let $d < \omega$ and let $R_0,\ldots,R_{d-1}$ be $\vec z$-robust subsets of 
	$D_\theta \cap (\p / H)$. 
	We show that there exists some $r \in \p / H$ such that for all $i < d$, $r \le \sum R_i$.

	By induction we construct a sequence of finite sequences 
	$$
	\langle \ \langle \ (p^i_\alpha,N^i_\alpha) \ : \ i < d \rangle \ : \ \alpha < \omega_1 \ \rangle
	$$ 
	so that the following are satisfied:
	\begin{enumerate}
		\item for all $\alpha < \omega_1$ and $i < d$, 
		$p^i_\alpha \in R_i$ and $N^i_\alpha \in \mathcal X(H)$;
		\item for all $\alpha < \omega_1$ and $i < d$, 
		$w_\theta(p_\alpha^i, N_\alpha^i) = \vec z$;
		\item for all $\alpha < \beta < \omega_1$ and 
		for all $i, j < d$, 
		$p_\alpha^i$ is in $\sk(N^j_\beta)$.
	\end{enumerate}
	Suppose that $\beta < \omega_1$ and for each $\alpha < \beta$ and each $i < d$, 
	$p^i_\alpha$ and $N^i_\alpha$ are defined. 
	Let $j < d$. 
	Using the fact that $R_j$ is $\vec z$-robust 
	we can pick some $p^j_\beta \in R_j$ and some $N^j_\beta \in \mathcal X(H)$ 
	such that $w_\theta(p^j_\beta,N^j_\beta) = \vec z$ and for all 
	$\alpha < \beta$ and $i < d$, $p^i_\alpha$ is in $\sk(N^j_\beta)$. 
	This completes the induction.

	For each $\alpha < \omega_1$ and $i < d$, write 
	$p_\alpha^i = (T_\alpha^i,W_\alpha^i,D_\alpha^i,A_\alpha^i)$ and 
	$\delta_\alpha^i = N_\alpha^i \cap \omega_1$. 
	Consider $\alpha < \beta < \omega_1$ and $i, j < d$. 
	By Lemma 8.7 applied to $p^i_\alpha$, $N^i_\alpha$, $p^j_\beta$, and $N^j_\beta$, 
	we have that $T_\alpha^i \cap T_\beta^j = T_\alpha^i \cap \delta_\alpha^i = 
	T_\beta^j \cap \delta_\beta^j$. 
	In particular, $T_\alpha^i \setminus \delta_\alpha^i$ and 
	$T_\beta^j \setminus \delta_\beta^j$ are disjoint.

	For each $\alpha < \omega_1$, define 
	$J_\alpha = \bigcup \{ T_\alpha^i \setminus \delta_\alpha^i : i < d \}$. 
	By the previous paragraph, for all $\alpha < \beta < \omega_1$, 
	$J_\alpha$ and $J_\beta$ are disjoint. 
	By the theorem of Baumgartner-Malitz-Reinhardt stated at the end of the introduction 
	and the fact that $T_H$ is Aronszajn, 
	fix $\alpha_0 < \cdots < \alpha_{d-1} < \omega_1$ so that for all 
	$i < j < d$, for all $x \in J_{\alpha_i}$ and for all $y \in J_{\alpha_j}$, 
	$x$ and $y$ are incomparable in $T_H$. 
	By Lemma 8.6, $q = p_{\alpha_0}^0 \oplus \cdots \oplus p_{\alpha_{d-1}}^{d-1}$ 
	is a condition which extends each of 
	$p_{\alpha_0}^0,\ldots,p_{\alpha_{d-1}}^{d-1}$. 
	Note that $q$ is in $E_\theta$. 
	By Lemma 11.6, $q$ is in $\p / H$. 
	For all $i < d$, $p_{\alpha_i}^i \in R_i$, 
	and hence $q \le p_{\alpha_i}^i \le \sum R_i$.
\end{proof}

\begin{proof}[Proof of Theorem 12.1]
	Working in $W$, fix a regular cardinal $\chi > \kappa$. 
	Define $\mathcal S$ to be the set of all $M \in [H(\chi)]^\omega$ 
	such that $M$ is an elementary substructure of the structure 
	$\mathcal B = (H(\chi),\in,\psi,\p,\theta,H,\p / H,D_\theta,w_\theta)$ 
	and $M \cap \kappa \in \mathcal X(H)$. 
	Since $\mathcal X(H)$ is stationary in $[\kappa]^\omega$, the set $\mathcal S$ 
	is stationary in $[H(\chi)]^\omega$.

	Let $M \in \mathcal S$ and define $N = M \cap \kappa$. 
	Consider $u \in M \cap (\p / H)$. 
	By extending $u$ further, we may assume that $u \in D_\theta$ and $A_u$ is non-empty. 
	By Theorem 12.4, 
	$u + N$ is in $\p / H$, is an extension of $u$, and is $(M,\p / H)$-generic.  
	Now consider a condition $q \le u + N$. 
	We will find a filter $\mathcal F$ on $\mathcal B(\p / H)$ in $M$ such that 
	for every $s \in M \cap \mathcal B(\p / H)$, if $q \le s$ then $s \in \mathcal F$.

	Using Proposition 9.4 and Lemma 11.1(4), 
	extend $q$ to some $r$ in $D_\theta \cap (\p / H)$ 
	such that $A_r$ is $N$-closed. 
	Then $(r, N)$ is in the domain of $w_\theta$. 
	Let $\vec z = w_\theta(q, N)$. 
	Then $\vec z \in \sk(N) \subseteq M$. 
	By Proposition 12.6, 
	the collection 
	$\mathcal F_0 = 
	\{ \sum R : R \subseteq D_\theta \cap (\p / H) \ \text{is $\vec z$-robust}\}$ 
	is a centered subset of $\mathcal B(\p / H)$. 
	By elementarity, $\mathcal F_0 \in M$. 
	So $\mathcal F = \{ b \in \mathcal B(\p / H) : \exists c \in \mathcal F_0 \ (c \le b) \}$ 
	is a filter on $\mathcal B(\p / H)$ which is also in $M$.
	
	To complete the proof, 
	suppose that $q \le s$ and $s \in M \cap \mathcal B(\p / H)$, 
	and we show that $s \in \mathcal F$. 
	Define $R = \{ t \in D_\theta \cap (\p / H) : t \le s \}$. 
	Clearly, $s = \sum R$, $R \in M$, and $r \in R$. 
	We claim that $R$ is $\vec z$-robust, and hence 
	$s = \sum R \in \mathcal F_0 \subseteq \mathcal F$. 
	Let $C$ be a club subset of $[\kappa]^\omega$ in $M$. 
	Then $N \in C$. 
	So $N \in C$, $N \in \mathcal X(H)$, $r \in R$, and $w_\theta(r, N) = \vec z$. 
	By elementarity, it follows that the set of $K \in \mathcal X(H)$ 
	for which there exists some $t \in R$ such that $w_\theta(t, K) = \vec z$ 
	is stationary in $[\kappa]^\omega$.
\end{proof}

\section{The Main Theorems: Part 2}

We are now prepared to prove the remaining main theorems of the article. 
We start by answering the 
question which originally motivated this work.

\begin{thm}
	The forcing poset $\p$ forces that there exists a strongly non-saturated 
	normal infinitely splitting Aronszajn tree and there does not 
	exist a weak Kurepa tree.
\end{thm}

\begin{proof}
	The first part was established in Theorem 10.13. 
	For the second part, suppose that $\dot T$ is a $\p$-name for 
	a tree with height and size $\omega_1$. 
	Without loss of generality, assume that $\dot T$ is forced to have 
	underlying set $\omega_1$ and $\dot T$ is a nice name. 
	Since $\p$ is $\kappa$-c.c., we can find some $\theta \in \Sigma$ 
	such that $\dot T$ is a nice $\p_\theta$-name. 
	By Lemma 11.2, $\p$ forces that $\dot T$ is in $V^{\p_\theta}$. 
	By Theorem 8.2 and Corollary 12.2, 
	$\p$ forces that every cofinal branch of $\dot T$ 
	in $V^\p$ lies in $V^{\p_\theta}$. 
	As $\kappa$ is inaccessible in $V^{\p_\theta}$ 
	and $\kappa$ equals $\omega_2$ in $V^\p$, $\p$ forces 
	that $\dot T$ has fewer than $\omega_2$-many cofinal branches.
\end{proof}

Since by a result of Solovay, the non-existence of a Kurepa tree implies that 
$\omega_2$ is inaccessible in $L$, 
we have the following corollary.

\begin{corollary}
	The following statements are equiconsistent.
	\begin{enumerate}
		\item There exists an inaccessible cardinal.
		\item There exists a strongly non-saturated Aronszajn tree and there does not 
		exist a weak Kurepa tree.
	\end{enumerate}
\end{corollary}

The next proposition follows from Theorem 8.2 and Corollary 12.2 by standard 
methods for constructing models satisfying the tree property (\cite{mitchell}).

\begin{proposition}
	If $\kappa$ is a Mahlo cardinal, then $\p$ forces that 
	there does not exist a special $\omega_2$-Aronszajn tree. 
	If $\kappa$ is a weakly compact cardinal, then $\p$ forces that there does not 
	exist an $\omega_2$-Aronszajn tree.
\end{proposition}

Recall the principle \textsf{ISP} of Weiss, which asserts the 
existence of an ineffable branch for every slender $P_{\omega_2}(\lambda)$-list, 
for any regular cardinal $\lambda \ge \omega_2$ (\cite{weiss}). 
Viale-Weiss \cite{vialeweiss} proved that this principle is equivalent to the 
statement that for any regular cardinal $\lambda \ge \omega_2$, 
there exist stationarily many guessing models in $[H(\lambda)]^{\omega_1}$. 
In \cite{coxkrueger2} we denote this last statement by \textsf{GMP}, 
which stands for the \emph{guessing model principle}.

For the next two results, we assume that the reader is familiar with constructing 
models of \textsf{GMP} (see, for example, \cite[Section 7]{coxkrueger}). 
In particular, if $\kappa$ is a supercompact cardinal, then by standard arguments 
the Y-properness of the quotient together with Theorem 8.2 imply 
the existence of stationarily many guessing models.

\begin{thm}
	Assuming that $\kappa$ is a supercompact cardinal, 
	$\p$ forces that \textsf{GMP} holds. 
	So the existence of a strongly non-saturated Aronszajn tree 
	is consistent with \textsf{GMP}.
\end{thm}

In \cite{coxkrueger2}, the idea of an indestructible guessing model is  
introduced together with the principle \textsf{IGMP}, which stands for the 
\emph{indestructible guessing model principle}. 
An \emph{indestructible guessing model} is 
a guessing model which remains guessing in any 
$\omega_1$-preserving generic extension. 
And \textsf{IGMP} states that 
for any regular cardinal $\lambda \ge \omega_2$, there exist stationarily many indestructible 
guessing models in $[H(\lambda)]^{\omega_1}$. 
By \cite[Corollary 4.5]{coxkrueger} and \cite[Theorem 1.4]{krueger}, 
\textsf{IGMP} follows from the conjunction of \textsf{GMP} 
and the statement that every tree of height and size $\omega_1$ which has no 
cofinal branches is special. 

We provide a sketch of a proof for how to use the indestructibility of the Y-properness 
of the quotient to obtain a model of \textsf{IGMP} together with a strongly 
non-saturated Aronszajn tree. 
For any tree $T$ with no uncountable branches, the standard forcing for specializing $T$ 
with finite conditions is Y-c.c.\ (\cite[Corollary 3.3]{CZ}). 
And any finite support forcing iteration of Y-c.c.\ forcings is Y-c.c.\ 
(\cite[Theorem 6.2]{CZ}). 
Consequently, there exists a Y-c.c.\ 
finite support forcing iteration of length 
$(2^{\omega_1})^+$ which forces that every tree with underlying set $\omega_1$ 
which has no cofinal branches is special.

Consider a generic filter $G$ on $\p$. 
In $V[G]$, we have that $2^{\omega_1} = \omega_2 = \kappa$, so we can 
fix a finite support forcing iteration 
$\langle \q_i : i \le \kappa \rangle$ as described in the previous paragraph. 
Let $K$ be a $V[G]$-generic filter on $\q_{\kappa}$. 
Consider any $\theta \in \Sigma$ such that $\langle \q_i : i \le \theta \rangle$ 
is in $V[G \cap \p_\theta]$. 
Let $G_\theta = G \cap \p_\theta$ and $K_\theta = K \cap \q_\theta$. 
By the product lemma, $V[G][K] = V[G_\theta][K_\theta][G][K]$. 
Now in $V[G_\theta]$, $\q_\theta$ is a finite support forcing iteration of Y-c.c.\ forcings, 
and hence is Y-c.c.\ and therefore Y-proper. 
So by Corollary 12.3, $\p / G_\theta$ is Y-proper on 
a stationary set in $V[G_\theta][K_\theta]$, and hence has the 
$\omega_1$-approximation property in $V[G_\theta][K_\theta]$ by Theorem 8.2. 
In $V[G_\theta][K_\theta][G] = V[G][K_\theta]$, 
$\q_\kappa / K_\theta$ is forcing equivalent to a finite support forcing iteration 
of Y-c.c.\ forcings, and hence has the $\omega_1$-approximation property 
by Theorem 8.2. 
It follows that $V[G][K]$ is a generic extension of $V[G_\theta][K_\theta]$ 
by the forcing $(\p / G_\theta) * (\q_\kappa / K_\theta)$ which has 
the $\omega_1$-approximation property.

The above analysis combined with standard methods 
shows that if $\kappa$ is a supercompact cardinal, then in $V[G][K]$ we have 
that \textsf{GMP} holds and every tree of height and size $\omega_1$ 
which has no cofinal branches is special.  
(More specifically, we apply the arguments of the previous paragraph to 
$j(\p)$, $j(\p) / G$, and $\theta = \kappa \in j(\Sigma)$, 
where $j : V \to M$ is an elementary embedding witnessing the supercompactness of $\kappa$ 
and $G$ is a generic 
filter on $\p = j(\p)_{\kappa}$.) 
So \textsf{IGMP} holds in $V[G][K]$.

\begin{thm}
	Suppose that $\kappa$ is a supercompact cardinal. 
	Then there exists a $\p$-name $\dot \q$ for a finite support forcing iteration 
	of Y-c.c.\ forcings of length $\kappa$ such that 
	$\p * \dot \q$ forces that there exists a strongly non-saturated Aronszajn tree 
	and \textsf{IGMP} holds.
\end{thm}

\providecommand{\bysame}{\leavevmode\hbox to3em{\hrulefill}\thinspace}
\providecommand{\MR}{\relax\ifhmode\unskip\space\fi MR }
\providecommand{\MRhref}[2]{%
  \href{http://www.ams.org/mathscinet-getitem?mr=#1}{#2}
}
\providecommand{\href}[2]{#2}


\end{document}